\documentclass[11pt]{article}
\usepackage{amsmath, amssymb, theorem, latexsym, epsfig, graphics, eufrak}
\usepackage{subfigure}
\numberwithin{equation}{section}

\theoremstyle{plain}
\theorembodyfont{\itshape}
\newtheorem{theorem}{Theorem}[section]
\newtheorem{proposition}[theorem]{Proposition}
\newtheorem{lemma}[theorem]{Lemma}
\newtheorem{corollary}[theorem]{Corollary}
  \newtheorem{conjecture}[theorem]{Conjecture}                                                                              
\theorembodyfont{\rmfamily}
\newtheorem{definition}[theorem]{Definition}

\newtheorem{remark}[theorem]{Remark}
\newtheorem{convention}[theorem]{Convention}
\newenvironment{proof}{{\noindent \textbf{Proof}\,\,}}{\hspace*{\fill}$\Box$\medskip}

\title{On constrictions of phase-lock areas in model of overdamped Josephson effect and transition matrix of 
double confluent Heun equation}
\author{A.A.Glutsyuk\thanks{ CNRS, France (UMR 5669 (UMPA, ENS de Lyon) and Interdisciplinary Scientific Center 
J.-V.Poncelet). 
Email: aglutsyu@ens-lyon.fr}
\thanks{National Research University Higher School of Economics (HSE), Moscow, Russia}
 \thanks{Supported by  RSF grant FF-18-41-05003.}}
\begin{document}
\maketitle
\def\zz{\mathbb Z}
\def\nn{\mathbb N}
\def\var{\varepsilon}
\def\td{\mathbb T^3} 
\def\rr{\mathbb R}
\def\la{\lambda}
\def\go#1{\EuFrak#1}
\def\wt{\widetilde}
\def\sign{\operatorname{sign}}
\def\cc{\mathbb C}
\def\oc{\overline\cc}
\def\diag{\operatorname{diag}}
\def\dd{\Delta_{discr}}
\def\mpp{\mathcal P}
\def\mca{\mathcal A}
\def\La{\Lambda}
\def\mcp{\mathcal P}
\def\mcl{\mathcal L}
\def\im{\operatorname{Im}}
\def\re{\operatorname{Re}}
\def\bbi{\mathbb I}
\def\mce{\mathcal E}

\begin{abstract} In 1973 B.Josephson received Nobel Prize for discovering a new fundamental effect concerning 
a {\it Josephson junction,} --  a system of two superconductors 
separated by a very narrow dielectric: there could exist a supercurrent tunneling through this junction.  We will discuss the model of the overdamped Josephson junction, which is given by a 
family of  first order non-linear ordinary differential equations on two-torus depending on three parameters: 
a fixed parameter $\omega$ (the {\it frequency}); a pair of variable parameters $(B,A)$ 
that are called respectively  the {\it abscissa,} and  the {\it ordinate.} 
It is important to study the rotation number of the system  as a function $\rho=\rho(B,A)$ and to describe the {\it phase-lock areas:} its level sets $L_r=\{\rho =r\}$  with non-empty interiors. 
They were studied by V.M.Buchstaber, O.V.Karpov, S.I.Tertychnyi, who observed  
in 2010  that the phase-lock areas exist only for integer values of the rotation number. It is known that each phase-lock area is a garland  of infinitely many bounded domains going to infinity in the vertical direction; each two subsequent domains are separated by one point, which is called {\it constriction} (provided that it does not lie in the abscissa axis). 
Those  points of intersection of the boundary $\partial L_r$ of the  phase-lock area $L_r$ 
with the line 
$\Lambda_r=\{ B=r\omega\}$ (which is called its {\it axis}) that are not constrictions are 
called {\it simple intersections.}  It is known that our family of dynamical systems is 
related to appropriate family of double confluent Heun equations with the same 
parameters via Buchtaber--Tertychnyi construction.  Simple intersections correspond to some of those parameter values for which the corresponding "conjugate" double confluent Heun equation has a polynomial solution 
(follows from results of a joint paper of  V.M.Buchstaber and S.I.Tertychnyi and a 
joint paper of V.M.Buchstaber and the author).  There is a conjecture stating that {\it all the constrictions of every phase-lock area $L_r$ lie in its axis $\Lambda_r$.} 
This conjecture was studied and partially proved in a joint paper of the author with V.A.Kleptsyn, D.A.Filimonov and 
I.V.Schurov. Another conjecture states that for  any two subsequent constrictions in $L_r$ with positive ordinates  
the interval between them also lies in $L_r$. In this paper we present  new results partially confirming both 
conjectures. The main result states that the phase-lock area $L_r$ contains the infinite interval of the axis $\La_r$ issued upwards from 
 the point of intersection $\partial L_r\cap\La_r$ with the biggest possible ordinate that is not 
a constriction. The proof is done by studying the complexification of the system under question, which is the projectivization 
of a family of systems of second  order linear equations with two  irregular non-resonant singular points at zero and at infinity. 
We obtain new results on the transition matrix between appropriate canonical solution bases of the linear system; on 
its behavior as a function of parameters. 
A key result, which implies the main result of the paper, states that 
the off-diagonal terms of the transition matrix are both non-zero 
at each constriction. We reduce the above conjectures on constrictions to the 
conjecture on negativity of the ratio of the latter off-diagonal terms at each constriction. 
\end{abstract}

\tableofcontents
\section{Introduction}

\subsection{Phase-lock areas in Josephson effect: history,  main conjectures and main results}
We study the family 
 \begin{equation}\frac{d\phi}{dt}=-\sin \phi + B + A \cos\omega t, \ \omega>0, \ B\geq0.\label{jos}\end{equation}
  of  {\it nonlinear} equations, which arises in several models in physics, mechanics and geometry. Our main 
  motivation is that it describes the overdamped model of the 
Josephson junction (RSJ - model)  in superconductivity, see \cite{josephson, stewart, mcc, bar, schmidt}. It 
 arises in  planimeters, see  \cite{Foote, foott}. 
Here $\omega$ is a fixed constant, and $(B,A)$ are the parameters. Set 
$$\tau=\omega t, \ l=\frac B\omega, \ \mu=\frac A{2\omega}.$$
The variable change $t\mapsto \tau$ transforms (\ref{jos}) to a 
non-autonomous ordinary differential equation on the two-torus $\mathbb T^2=S^1\times S^1$ with coordinates 
$(\phi,\tau)\in\rr^2\slash2\pi\zz^2$: 
\begin{equation} \dot \phi=\frac{d\phi}{d\tau}=-\frac{\sin \phi}{\omega} + l + 2\mu \cos \tau.\label{jostor}\end{equation}
The graphs of its solutions are the orbits of the vector field 
\begin{equation}\begin{cases} & \dot\phi=-\frac{\sin \phi}{\omega} + l + 2\mu \cos \tau\\
& \dot \tau=1\end{cases}\label{josvec}\end{equation}
on $\mathbb T^2$. The {\it rotation number} of its flow, see \cite[p. 104]{arn},  is a function $\rho(B,A)$ of parameters\footnote{There is a misprint, 
missing $2\pi$ in the denominator, in analogous formulas in previous papers of the 
author with co-authors: \cite[formula (2.2)]{4}, \cite[the formula after (1.16)]{bg2}.}:
$$\rho(B,A;\omega)=\lim_{k\to+\infty}\frac{\phi(2\pi k)}{2\pi k}.$$
Here $\phi(\tau)$ is a general $\rr$-valued solution of the first equation in (\ref{josvec}) 
whose parameter is the initial condition 
for $\tau=0$. Recall that the rotation number is independent on the choice of the initial condition, see \cite[p.104]{arn}. 
 The parameter $B$ is called {\it abscissa,} and $A$ is called the {\it ordinate.} 
 Recall the following well-known definition. 

\begin{definition} (cf. \cite[definition 1.1]{4}) The {\it $r$-th phase-lock area} is the level set 
$$L_r=\{\rho(B,A)=r\}\subset\rr^2,$$ 
provided that it has a non-empty interior. 
\end{definition}

\begin{remark}{\bf: phase-lock areas and Arnold tongues.} H.Poincar\'e introduced the rotation number of a circle 
diffeomorphism. The rotation number of the flow  of the field (\ref{josvec}) on $\mathbb T^2$ equals 
(modulo $\zz$) the rotation number of the circle diffeomorphism given by its time $2\pi$ flow mapping 
restricted to the cross-section $S^1_{\phi}\times\{0\}$. In Arnold family of circle diffeomorphisms 
$x\mapsto x+b+a\sin x$, $x\in S^1=\rr\slash2\pi\zz$ the behavior of its phase-lock areas for small $a$ demonstrates the tongues 
effect discovered by V.I. Arnold \cite[p. 110]{arn}. That is why the phase-lock areas became 
 ``Arnold tongues'', see  \cite[definition 1.1]{4}. 
\end{remark}

Recall that the rotation number has physical meaning of the mean voltage over a long time interval. 
The phase-lock areas of the family (\ref{jostor}) were studied by V.M.Buchstaber, O.V.Karpov, S.I.Tertychnyi et al, see \cite{bg}--\cite{bt3}, \cite{RK}, \cite{4} and references therein. It is known that the 
following statements hold: 

1) Phase-lock areas exist only for integer values of the rotation number  
(a ``quantization effect'' observed in \cite{buch2} and later also proved 
in  \cite{IRF, LSh2009}).

2) The boundary of the $r$-th phase-lock area 
 consists of two analytic curves, which are the graphs of two
functions $B=g_{r,\pm}(A)$ (see \cite{buch1}; this fact was later explained by A.V.Klimenko via symmetry, see \cite{RK}).

3)  The latter functions have Bessel asymptotics
\begin{equation}\begin{cases} g_{r,-}(s)=r\omega-J_r(-\frac s\omega)+O(\frac{\ln |s|}s) \\ 
g_{r,+}(s)=r\omega+J_r(-\frac s\omega)+O(\frac{\ln |s|}s)\end{cases}, \text{ as } s\to\infty\label{bessas}\end{equation}
 (observed and proved on physical level in ~\cite{shap}, see also \cite[chapter 5]{lich},
 \cite[section 11.1]{bar}, ~\cite{buch2006}; proved mathematically in ~\cite{RK}).
 
 4) Each phase-lock area is a garland  of infinitely many bounded domains going to infinity in the vertical direction. 
 In this chain each two subsequent domains are separated by one point. This follows from the 
 above statement 3). Those of the latter separation 
 points that lie in the horizontal $B$-axis are calculated explicitly,  and 
 we call them the {\it growth points}, see \cite[corollary 3]{buch1}. The other separation points, which  lie outside the horizontal $B$-axis, are called the  {\it constrictions}. 
 
 5)  For every $r\in\zz$ the $r$-th phase-lock area is symmetric to the $-r$-th one with respect to the vertical 
 $A$-axis.
  
6) Every phase-lock area is symmetric with respect to the horizontal $B$-axis. See Figures 1--5 below.

\begin{figure}[ht]
  \begin{center}
   \epsfig{file=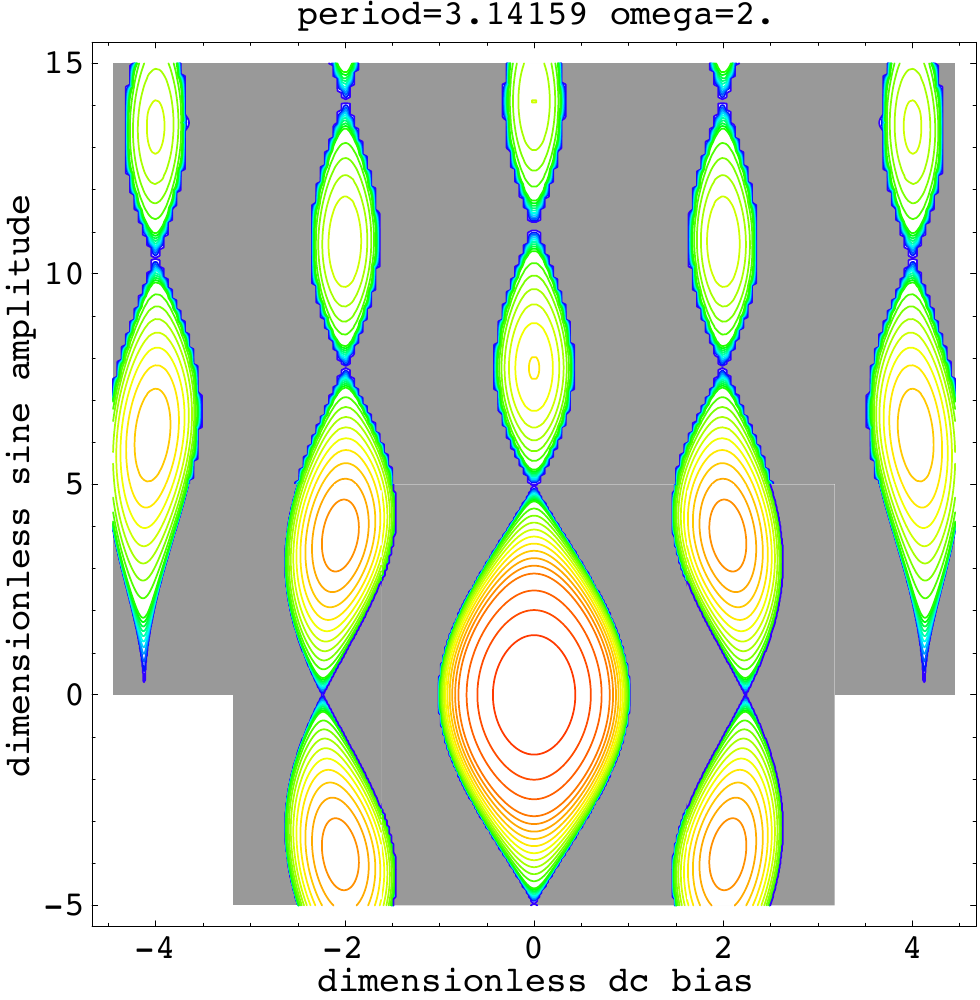}
    \caption{Phase-lock areas and their constrictions for $\omega=2$. The abscissa is $B$, the ordinate is $A$. Figure 
    taken from \cite[fig. 1a)]{bg2}}
  \end{center}
\end{figure} 

\begin{figure}[ht]
  \begin{center}
   \epsfig{file=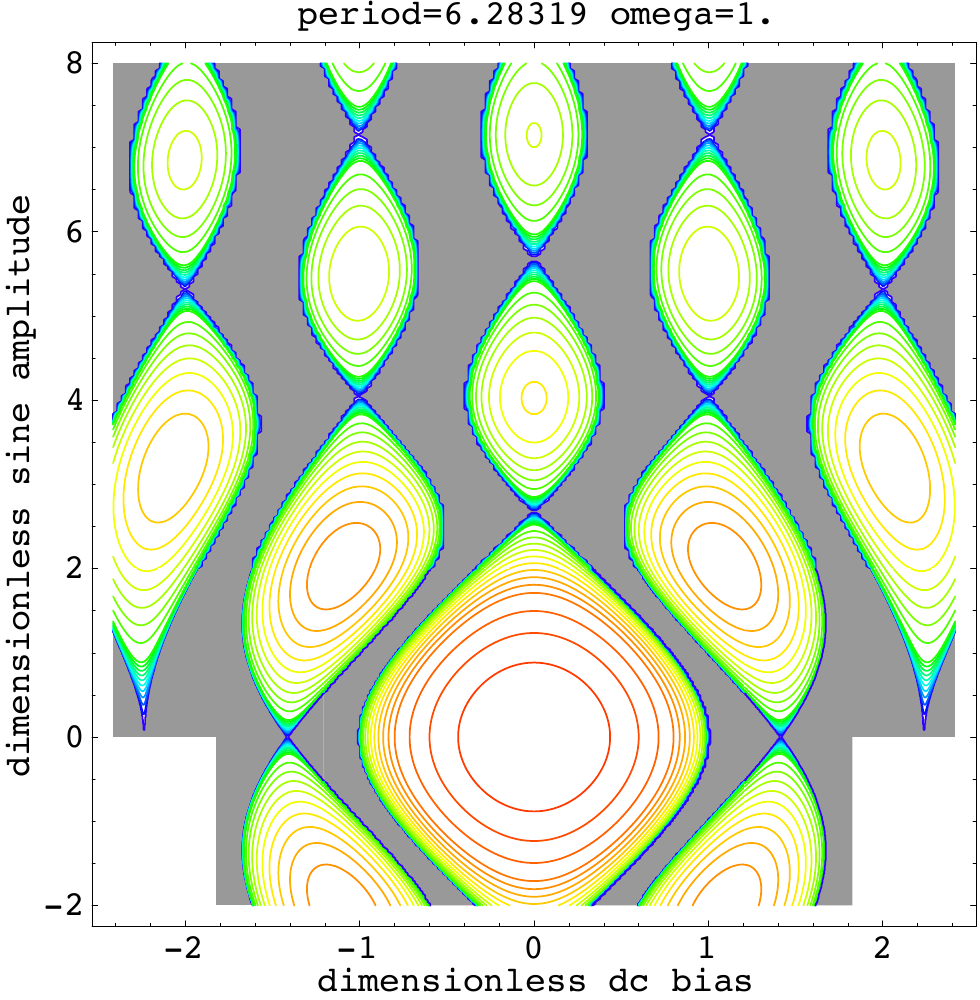}
    \caption{Phase-lock areas and their constrictions  for $\omega=1$. The abscissa is $B$, the ordinate is $A$. 
    Figure 
    taken from \cite[fig. 1b)]{bg2}}
     \end{center}
\end{figure} 
 
 \begin{figure}[ht]
  \begin{center}
   \epsfig{file=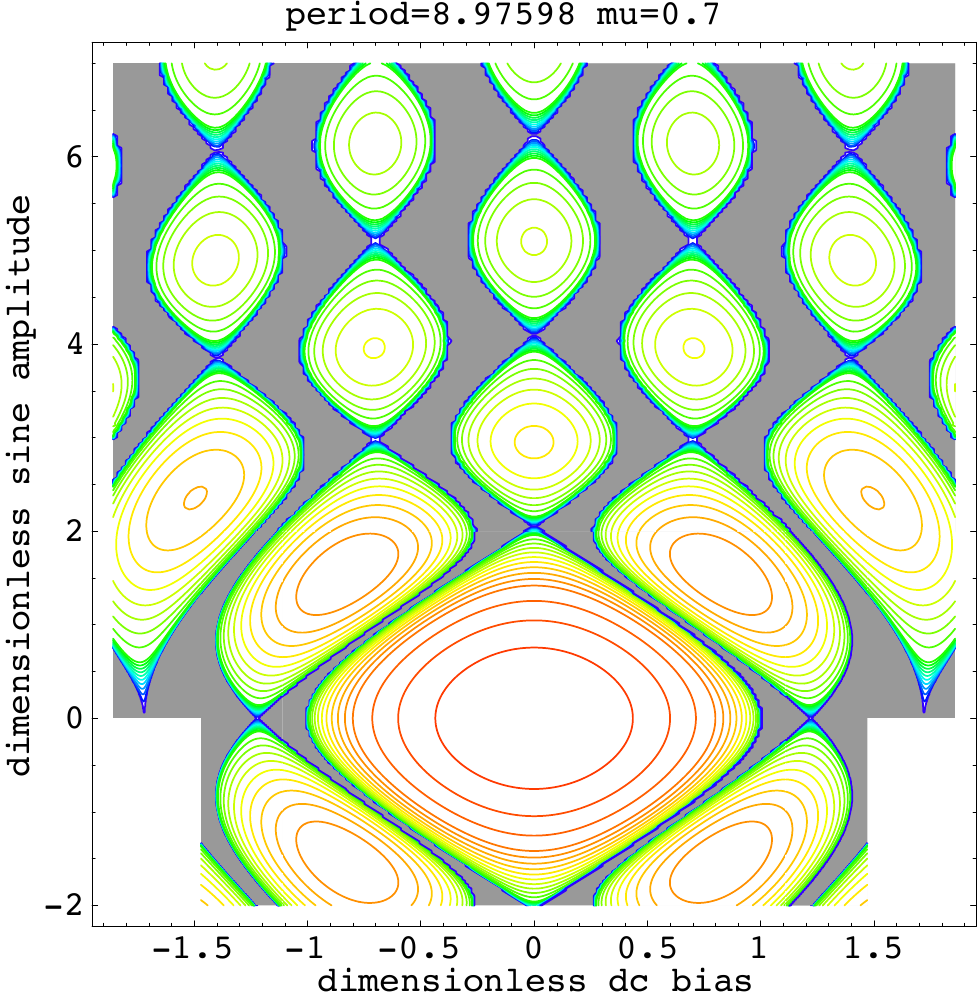}
    \caption{Phase-lock areas and their constrictions  for $\omega=0.7$. Figure taken from \cite[p. 331]{bt1}, 
    see also \cite[fig. 1c)]{bg2}.}
    \end{center}
\end{figure} 

 \begin{figure}[ht]
  \begin{center}
   \epsfig{file=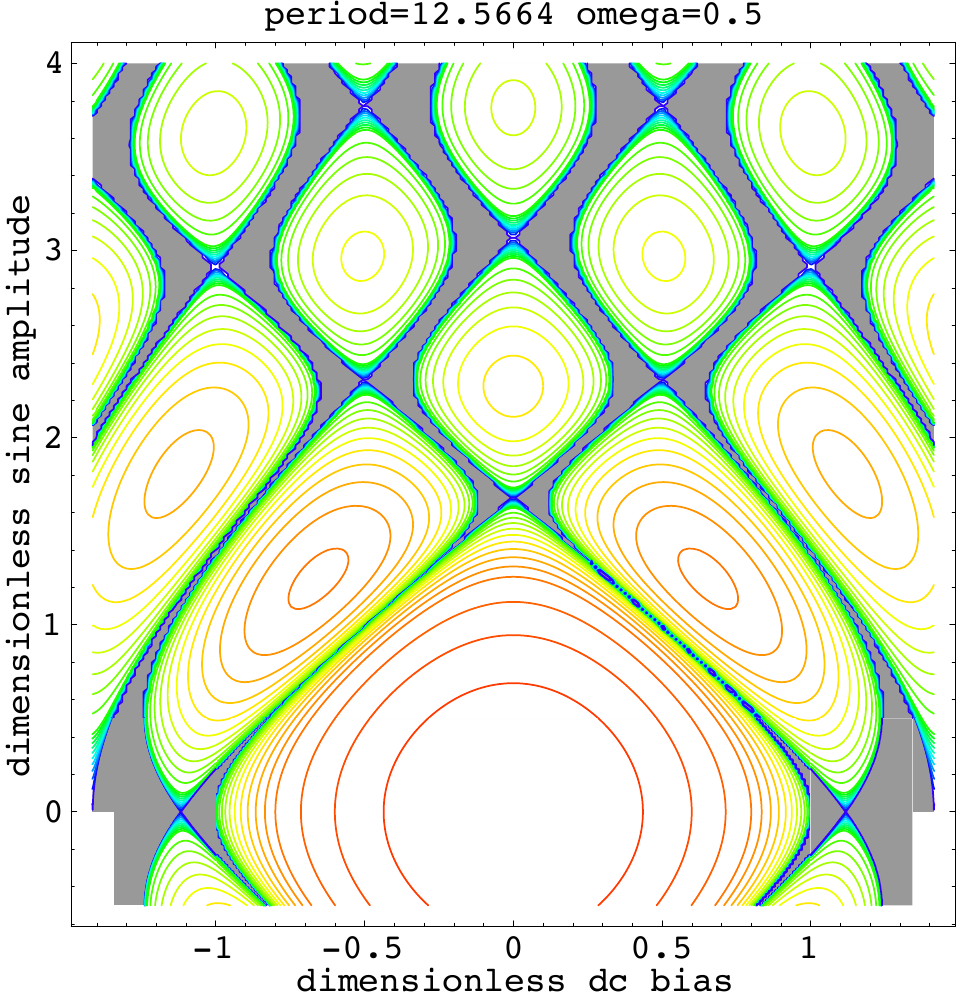}
    \caption{Phase-lock areas and their constrictions  for $\omega=0.5$. Figure 
    taken from \cite[fig. 1d)]{bg2}}
     \end{center}
\end{figure} 

\begin{figure}[ht]
  \begin{center}
   \epsfig{file=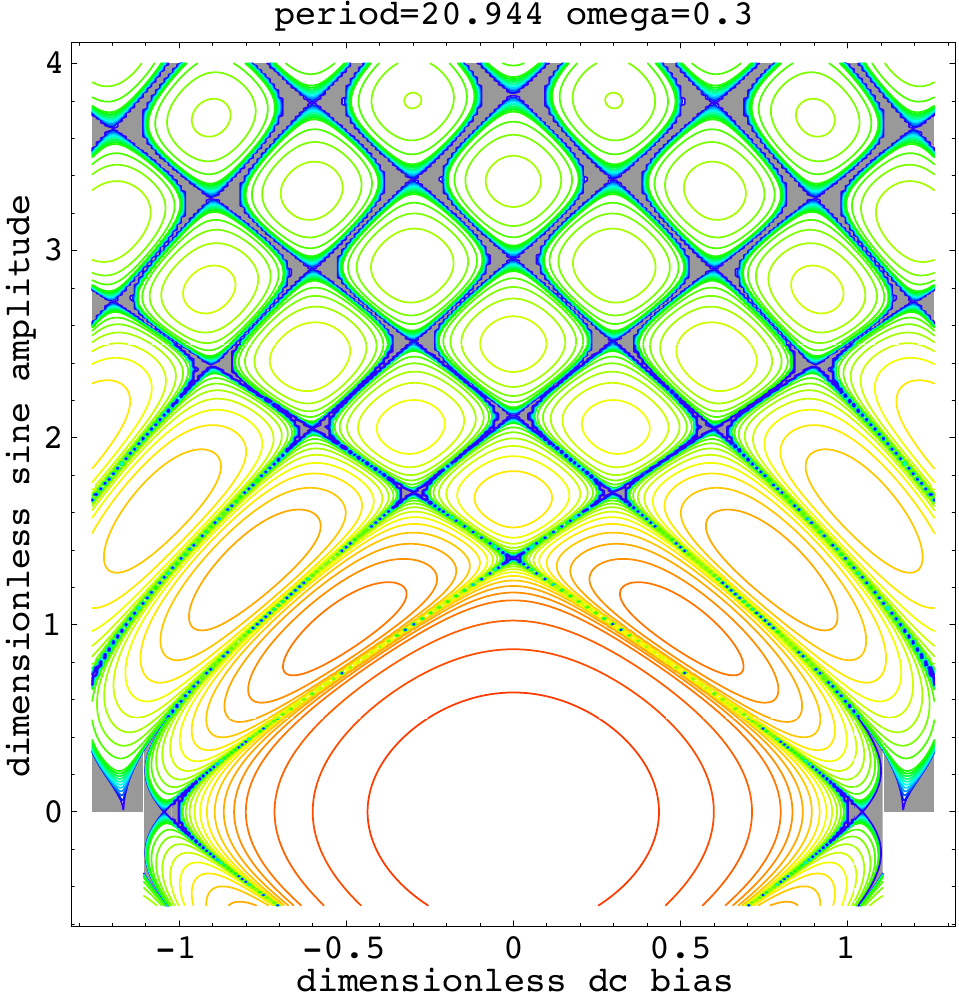}
    \caption{Phase-lock areas and their constrictions  for $\omega=0.3$. Figure 
    taken from \cite[fig. 1e)]{bg2}}
     \end{center}
\end{figure}

\begin{definition} For every $r\in\zz$ and $\omega>0$ we consider the vertical line  
$$\Lambda_r=\{ B=\omega r\}\subset\rr^2_{(B,A)}$$
and we will call it the {\it axis} of the phase-lock area $L_r$. 
\end{definition}

Numerical experiences made by V.M.Buchstaber, S.I.Tertychnyi, V.A.Kleptsyn, D.A.Filimonov, 
I.V.Schurov led to the following conjecture, which was stated and partially investigated in \cite{4}, see also 
\cite[section 5]{bg2}.

\begin{conjecture} \label{garl} (\cite[experimental fact A]{4}, \cite[conjecture 5.17]{bg2}).  
The upper part $L_r^+=L_r\cap\{ A\geq0\}$ of each phase-lock area $L_r$ 
is a garland of infinitely many connected components separated by constrictions 
$\mathcal A_{r,1}, \mathcal A_{r,2}\dots$ lying in its axis $\La_r=\{ B=r\omega\}$ and 
ordered by their ordinates  $A$, see the figures below. 
\end{conjecture} 

\begin{remark} \label{ssylka-4} Conjecture \ref{garl} was proved in \cite{4} for $\omega\geq1$. It was proved in the same 
paper that for every $\omega>0$ all the constrictions of every phase-lock area $L_r$ have abscissas 
$B=\omega l$, $l\in\zz$, $l\equiv r(mod2)$, $l\in[0,r]$. It is known that the zero phase-lock 
area $L_0$ contains the whole $A$-axis $\Lambda_0$, and all its constrictions lie in 
$\Lambda_0$; each point of the intersection $\partial L_0\cap\Lambda_0$ is a constriction. This follows from symmetry of the phase-lock area $L_0$  with respect to 
its axis $\Lambda_0$ and the fact that the  interval $(-1,1)$ of the $B$-axis is 
contained in $Int(L_r)$, see \cite[proposition 5.22]{bg2}. The two above 
statements together imply that for $r=\pm1,\pm2$ all the constrictions of the phase-lock area $L_r$ 
lie in its axis $\La_r=\{ B=r\omega\}$. But it is not known whether the latter 
statement holds for every $r\in\zz\setminus\{0\}$. 
\end{remark} 

\begin{conjecture} \label{garl1} \cite[conjecture 5.19]{bg2} $[\mca_{r,j},\mca_{r,j+1}]\subset L_r$ for every $r\in\zz$ and $j\in\nn$.
\end{conjecture}

\begin{conjecture} \label{right}  (see also \cite[conjecture 5.26]{bg2}) 
 Each phase-lock area $L_r$ with $r\in\nn$ lies on the right from the  axis $\La_{r-1}$: that is, 
$B|_{L_r}>(l-1)\omega$.
\end{conjecture}

The main results of the present paper are Theorems \ref{locon} and \ref{intcon} 
stated below, which are partial results 
towards confirmation of Conjectures 
\ref{garl} and \ref{garl1}.

\begin{theorem} \label{locon} For every $\omega>0$ and every constriction $(B_0,A_0)$ there exists a punctured 
neighborhood $U=U(A_0)\subset\rr$ such that the punctured interval $B_0\times (U\setminus\{A_0\})\subset B_0\times\rr$ 
either lies entirely in the interior of  a phase-lock area (then the constriction is called {\bf positive}), or lies entirely outside 
the union of the phase-lock areas (then the constriction is called {\bf negative}), see Fig.6. 
\end{theorem}

\begin{figure}[ht]
  \begin{center}
   \epsfig{file=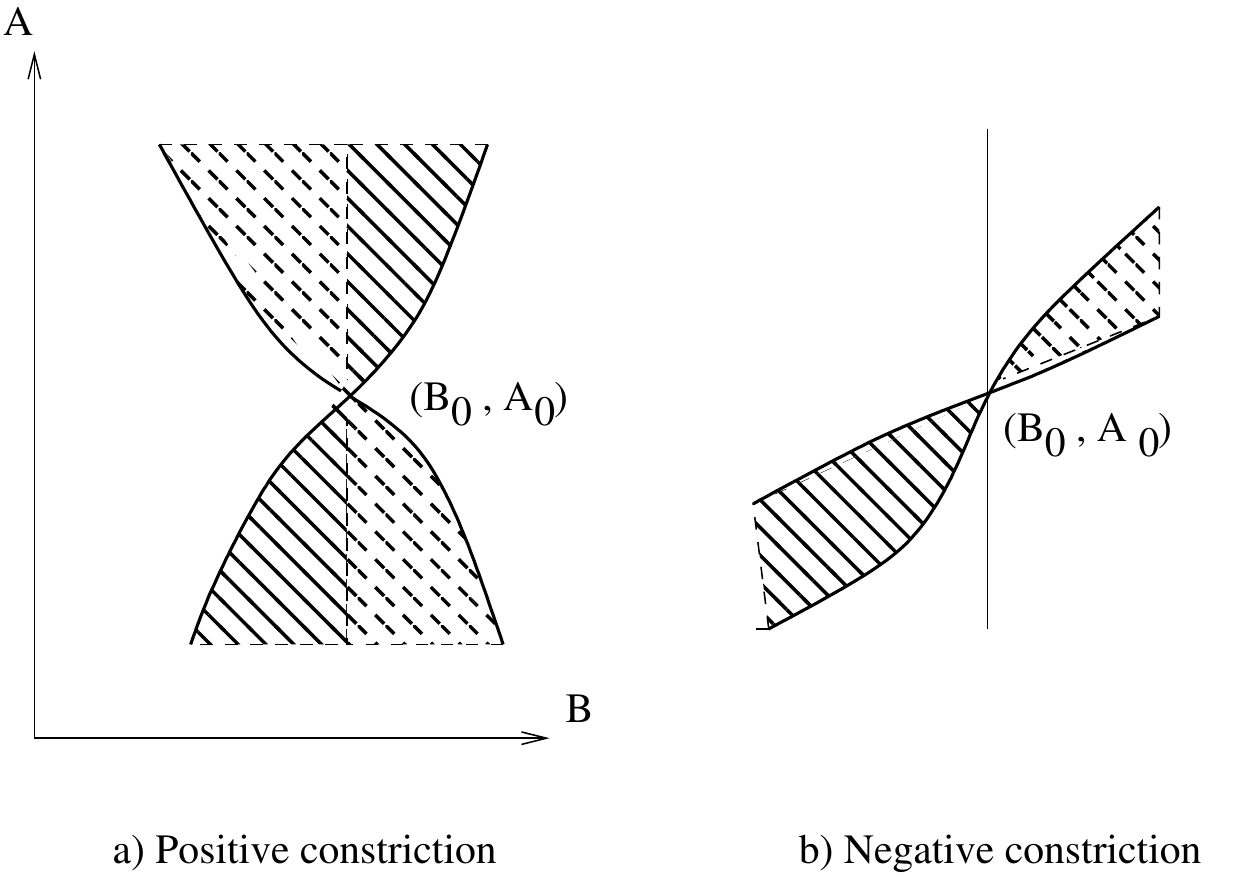}
    \caption{Positive and (conjecturally non-existing) negative constrictions. The shaded domains are phase-lock areas}
  \end{center}
\end{figure} 

To state the second theorem, let us introduce the following definition. 

\begin{definition} \label{simpint}  A {\it simple intersection} is a point of intersection of the boundary $\partial L_r$ 
with the axis 
$\Lambda_r$ that is not a constriction. The simple intersection with the maximal 
ordinate  $A$ will be 
called the {\it higher simple intersection} and denoted by $\mpp_r$. Set 
$$Sr=\{\omega r\}\times[A(\mcp_r),+\infty)\subset\Lambda_r:$$ 
this is the vertical ray in $\Lambda_r$ issued from the point $\mpp_r$ in the direction of increasing of the coordinate $A$. See  Fig.7. 
\end{definition}

\begin{conjecture} \label{cunique} 
For every $r\in\zz\setminus\{0\}$ the simple intersection with positive ordinate is unique.
\end{conjecture}

\begin{figure}[ht]
  \begin{center}
   \epsfig{file=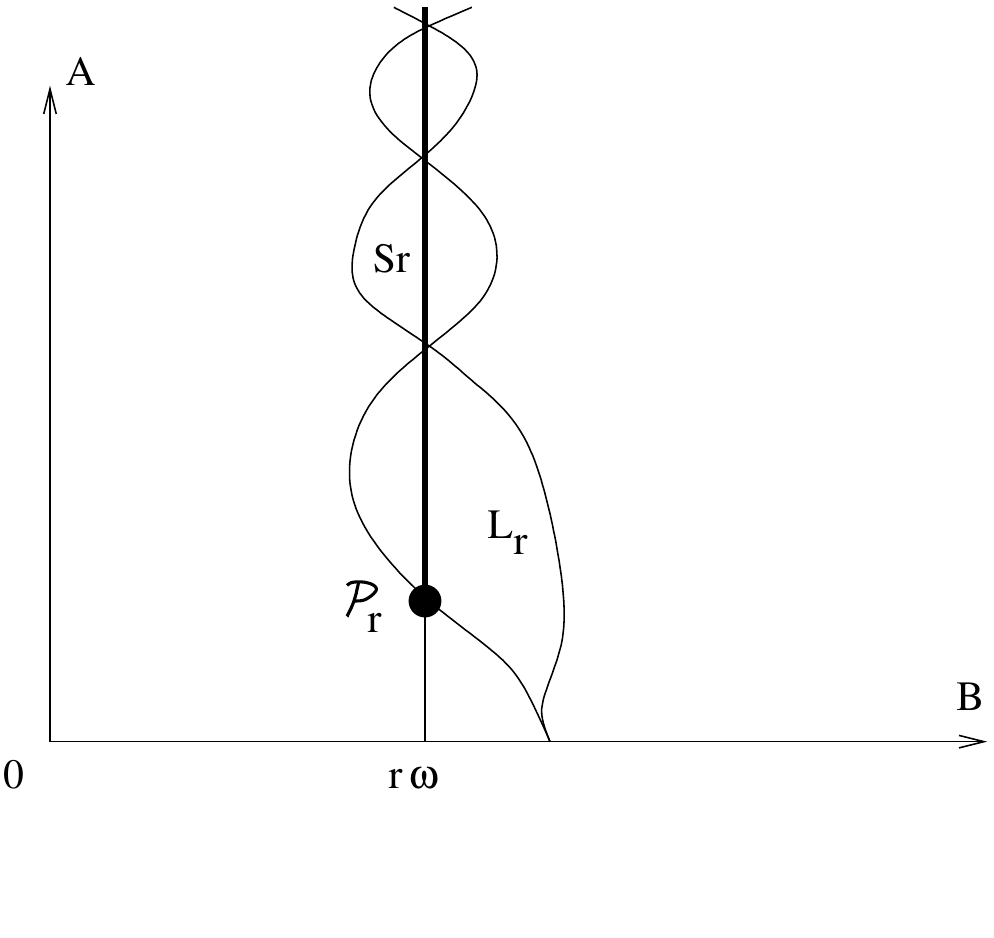}
    \caption{The higher simple intersection $\mcp_r$ and the corresponding ray $Sr$}
  \end{center}
\end{figure} 

\begin{remark} There are no simple intersections for $r=0$, since 
the intersection $\partial L_0\cap\Lambda_0$ consists only of constrictions, see Remark 
\ref{ssylka-4}.  
For every $r\in\zz\setminus\{0\}$ the ordinates of the 
simple intersections in $\Lambda_r$ lying in the upper half-plane
 belong to the  
collection of roots of a known polynomial of degree $|r|$. 
This follows from results of \cite[section 3]{bt0} and \cite[theorem 1.15]{bg2}. This 
together with symmetry implies that 
for every given $r\in\zz\setminus\{0\}$ the number of the corresponding simple intersections is finite. 
\end{remark}

\begin{theorem} \label{intcon} For every $\omega>0$ and every 
$r\in\zz\setminus\{0\}$ the corresponding simple intersections exist and 
do not lie in the $B$-axis; thus $\mcp_r$ and $Sr$  are well-defined. 
 The ray $Sr$ is contained in the phase-lock area $L_r$. 
\end{theorem} 

\begin{conjecture} \label{cocon} All the constrictions are positive. 
\end{conjecture}

\begin{conjecture} \label{coinc} The intersection $L_r^+\cap\La_r$ coincides with $Sr$ 
for every $r\in\zz\setminus\{0\}$. 
\end{conjecture} 

\begin{remark} Conjecture \ref{coinc} obviously implies Conjecture \ref{cunique}. 
In Section 4 we show that any of Conjectures \ref{coinc}, \ref{cocon} implies Conjectures \ref{garl} 
and \ref{garl1}, and we will discuss the relations between different conjectures in more details. 
\end{remark}

Theorem \ref{intcon} will be deduced from Theorem \ref{locon} and the result of paper \cite{RK}. 
For the proof of Theorem \ref{locon} we complexify equation (\ref{jos}) and write it in the new complex variables 
$$\Phi=e^{i\phi}, \ z=e^{i\tau}=e^{i\omega t},$$
set
\begin{equation} l=\frac B\omega,  \ \mu=\frac A{2\omega}, \ \lambda=\left(\frac1{2\omega}\right)^2-\mu^2.\label{param}\end{equation}
The complexified equation (\ref{jos}) becomes the Riccati equation 
\begin{equation}\frac{d\Phi}{dz}=z^{-2}((lz+\mu(z^2+1))\Phi-\frac z{2i\omega}(\Phi^2-1)).\label{ric}\end{equation}
The latter is the projectivization of 
the following linear equation on a vector function $(u,v)$, $\Phi=\frac v{u}$:
\begin{equation}\begin{cases} & u'=z^{-2}(-(lz+\mu(1+z^2))u+\frac z{2i\omega}v)\\
& v'=\frac1{2i\omega z}u\end{cases}:\label{tty}\end{equation}
each solution $\Phi(z)$ of equation (\ref{ric}) is a ratio $\frac vu$ of components of a solution of equation 
(\ref{tty}) and vice versa. See \cite[sect. 3.2]{bg}. The above reduction to a linear system  was obtained in slightly different terms in \cite{bkt1, bt1, Foote, IRF}.

After substitution $E(z)=e^{\mu z}v(z)$ system (\ref{tty}) becomes 
equivalent to the following {\it special  double confluent Heun equation:} 
\begin{equation} z^2E''+((l+1)z+\mu(1-z^2))E'+(\lambda-\mu(l+1)z)E=0.\label{heun}\end{equation}
Equation (\ref{heun}) belongs to the well-known class of double confluent Heun 
equations, see \cite[formula (3.1.15)]{sla}. The reduction to equations (\ref{heun}) was 
obtained by V.M.Buchstaber and S.I.Tertychnyi \cite{bt0, bt1, tert, tert2}, who studied 
equations (\ref{heun})  and obtained many important results on their 
polynomial and entire solutions and symmetries in loc. cit. and in \cite{bt2, bt3, bt4}). 
The complete description of equations (\ref{heun}) having entire solutions was 
started in \cite{bt1} and finished in \cite{bg}. The description of their monodromy 
eigenvalues was obtained in the joint paper \cite{bg2} of V.M.Buchstaber and the author 
 as an  explicit analytic transcendental equation relating one monodromy eigenvalue and the parameters. The following theorem was also proved in \cite{bg2}. It concerns the "conjugate" double confluent 
 Heun equations  
 \begin{equation} z^2E''+((-l+1)z+\mu(1-z^2))E'+(\lambda+\mu(l-1)z)E=0\label{heun2}\end{equation}
obtained from (\ref{heun}) by changing sign at the parameter $l$.

 \begin{theorem} \label{thpol2} \cite[theorem 1.15]{bg2}.  Let $\omega>0$, $(B,A)\in\rr^2$, $B,A>0$, $l=\frac B\omega$, $\mu=\frac A{2\omega}$, 
$\la=\frac1{4\omega^2}-\mu^2$, $\rho=\rho(B,A)$. The "conjugate" double confluent Heun equation (\ref{heun2}) corresponding to the above  
$\la$, $\mu$ and $l$ has a polynomial solution, if and only if $l,\rho\in\zz$, $\rho\equiv l(mod 2\zz)$, $0\leq\rho\leq l$, 
the point $(B,A)$ lies in the boundary of a phase-lock area  and is not a constriction. In other terms, 
the points $(B,A)\in\rr_+^2$ corresponding to equations (\ref{heun2}) with polynomial solutions lie in boundaries of 
phase-lock areas and  are exactly those their intersection points with the lines 
$\{ B=m\omega\}$, $m\equiv\rho(mod 2\zz)$, $0\leq\rho\leq m$ 
that are not  constrictions. (For example, the statement of the theorem holds for  simple intersections.) 
\end{theorem}

System (\ref{tty}) is a holomorphic linear differential equation 
on the Riemann sphere with two irregular non-resonant singular points of Poincar\'e rank 1 (pole of order 2) 
at zero and at infinity. The classical Stokes phenomena theory 
 \cite{2, 12, bjl, jlp, sib} yields canonical bases of its
 solutions  at 0 and 
at $\infty$  in two appropriate sectors $S_{\pm}$  containing the punctured closed half-planes 
 $\{\pm\im z\geq0\}\setminus\{0\}$ and not  containing the opposite imaginary semiaxes $i\rr_{\mp}$. 
 In each sector we have two canonical solution bases: one comes from zero, and the other one comes from infinity. 
 The classical Stokes matrices at zero (infinity) compare appropriately normalized sectorial bases at zero (infinity) 
 on components of the intersection $S_+\cap S_-$. It is well-known that the Stokes matrices are 
 triangular and unipotent; their triangular elements $c_0$ and $c_1$ 
 are called the {\it Stokes multipliers.} In Subsection 2.2 we show that  the Stoker 
 multipliers "at zero" are real, whenever $l\in\zz$. 
 
The statement of Theorem \ref{locon} deals with a   constriction $(B_0,A_0)$, set $l=\frac{B_0}{\omega}\in\zz$, 
and says that its appropriate punctured neighborhood $U$ in the line $\{ B=l\omega\}$ either entirely lies in a phase-lock area, or entirely lies outside the union of the 
 phase-lock areas. Inclusion into a phase-lock area is equivalent to the statement that the  
 monodromy operator of system (\ref{tty}), which is unimodular, has trace with modulus no less than 2. 
The trace of monodromy under question equals $2+c_0c_1$, by the classical formula expressing the monodromy 
via the formal monodromy (which is trivial, since $l=\frac{B_0}{\omega}\in\zz$) and the Stokes matrices.  To prove Theorem \ref{locon}, we have to show that 
the product $c_0c_1$ has constant sign on appropriate punctured neighborhood $U$. The proof of this 
statement is based on studying of 
 two appropriately normalized canonical solution bases and the transition matrix comparing them: 
 the base on $S_+$ coming from zero, 
 and the base on $S_-$ coming from infinity. The idea to study a transition matrix between 
 canonical bases at zero and at infinity was suggested by V.M.Buchstaber. 
 
 We compare appropriate solution bases "at zero" on $S_+$ and "at infinity" on $S_-$ 
 that are given by two fundamental solution matrices 
 $W_+(z)$ and $\hat W_-(z)$ respectively. On the positive real semiaxis $\rr_+\subset S_+\cap S_-$ the 
 transition between these bases is defined by a constant matrix $Q$: 
 $$W_+(z)=\hat W_-(z)Q.$$
  We show that one can normalize the bases $W_+$ and $\hat W_-$ under 
 question so that $Q$ be an involution, and we prove  a formula relating the coefficients of the 
 matrix $Q$ and the Stokes multipliers. The key Lemma \ref{lkey} says that the off-diagonal terms of the 
 transition matrix are both non-zero at $(B_0,A_0)$. The latter formula and inequality together will imply that the 
 ratio $\frac{c_1}{c_0}$ is a function analytic and nonvanishing 
 on a neighborhood of the point $(B_0,A_0)$ in the line $\{ B=B_0\}$. This will imply constance of sign of 
 their product $c_0c_1$ on the punctured neighborhood. 
 
In Section 4 we state Conjecture \ref{posit} saying that the ratio of the off-diagonal elements of the transition 
 matrix $Q$ is negative at each constriction $(B_0,A_0)$ with $B_0\geq0$, $A_0>0$. We show that 
 Conjecture \ref{posit} would imply Conjectures \ref{garl} and \ref{garl1} and discuss further relations between different 
 conjectures. 
 
In Section 5 we present additional technical results on the coefficients of the transition matrix $Q$, which will 
be used further on. The main result of Section 5  (Theorem \ref{triq}) 
states that the upper triangular element of the matrix $Q$ 
is purely imaginary, whenever $l\in\zz_{\geq0}$ and $A>0$.

\begin{remark} Very recently Yulia Bibilo suggested a new approach to study 
the model of Josephson effect: to include system (\ref{tty}) 
into a general family of two-dimensional linear systems  with two irregular singularities 
at zero and infinity 
and study  isomonodromic conditions in the general family. Using this method,  
she have shown that an infinite collection of constrictions can be 
described as poles of Bessel solution of Painleve 3 equation with a special choice of parameters \cite{bibilo}. 
\end{remark} 
 
\begin{convention} \label{conv1}
 In what follows, whenever the contrary is not specified, 
 we consider that $l=\frac B{\omega}\geq0$ and $\mu=\frac A{2\omega}>0$:
 it suffices to treat the case of non-negative 
$l$ and positive $\mu$, by symmetry of the portrait of the phase-lock areas. 
\end{convention}

\section{Linearization. Stokes and transition matrices and their relation to the rotation number}

Here we study family (\ref{tty}) of linear systems equivalent to equations (\ref{jos}). 
In Subsection 2.1 we recall what are their formal normal forms, canonical sectorial solution bases  and  Stokes matrices 
and multipliers.   In Subsection 2.2 we show that the Stokes multipliers are real. 
In Subsection 2.3 we recall results on symmetries of equation (\ref{heun}) and existence 
of its polynomial and entire solutions. In Subsection 2.4 we introduce the transition matrix between 
appropriate sectorial bases at zero and at infinity and prove a preliminary formula relating its coefficients to 
the Stokes multipliers. 

\subsection{Preliminaries: canonical solution bases, Stokes and transition matrices}

All the results presented in this subsection are  particular cases of  classical results 
contained in \cite{2, 12, bjl, jlp, sib}. 

\begin{definition} Two germs at 0 of linear systems of type 
$$\dot w=\frac{A(z)}{z^{k+1}}w, \ w\in\cc^n,$$
are {\it analytically (formally) equivalent,} if there exists a linear variable change $w=H(z)\wt w$ with 
$H(z)$ being a holomorphic invertible matrix function (respectively, a formal invertible matrix power series in $z$) that transforms one equation to the other. 
\end{definition}

The germs at  zero and at infinity of linear system (\ref{tty}) are both formally  equivalent to the germ of the diagonal system 
\begin{equation} \begin{cases} & \wt u'=-z^{-2}(lz+\mu(1+z^2))\wt u\\
& \wt v'=0,\end{cases}\label{ttyn}\end{equation}
which will be here called the {\it formal normal form.} The formal normal form has the canonical  base in its 
solution space with the diagonal fundamental matrix 
\begin{equation}
F(z)=\left(\begin{matrix} & z^{-l}e^{\mu(\frac1z-z)}  & 0\\
& 0 & 1 \end{matrix}\right).\label{diagmat}\end{equation}
In general, the formal equivalence is not analytic: the normalizing power series diverges. On the other hand, there 
exist analytic normalizations on appropriate sectors. Namely, let 
  $S_+$, $S_-$ be  sectors on the $z$-axis with vertex at 0 that contain the closed upper (respectively, 
lower) half-plane punctured at 0 so that the closure of the sector  $S_{\pm}$ does not contain the 
opposite imaginary semiaxis $i\rr_{\mp}$, see Fig. 8. In addition we consider that the sector $S_{\pm}$ is 
symmetric to $S_{\mp}$ with respect to the real axis. Set 
$$w=(u,v), \ \wt w=(\wt u,\wt v).$$
There exist and unique  invertible matrix 
functions $H_{\pm}(z)$ that are holomorphic on the sectors $S_{\pm}$ and $C^{\infty}$ on their closures 
$\overline S_{\pm}\subset\oc$ punctured at $\infty$, 
$H_{\pm}(0)=Id$, such that the variable change $w=H_{\pm}(z)\wt w$ transforms system (\ref{tty}) to 
its formal normal form (\ref{ttyn}) (the Sectorial Normalization Theorem). 
\begin{figure}[ht]
  \begin{center}
   \epsfig{file=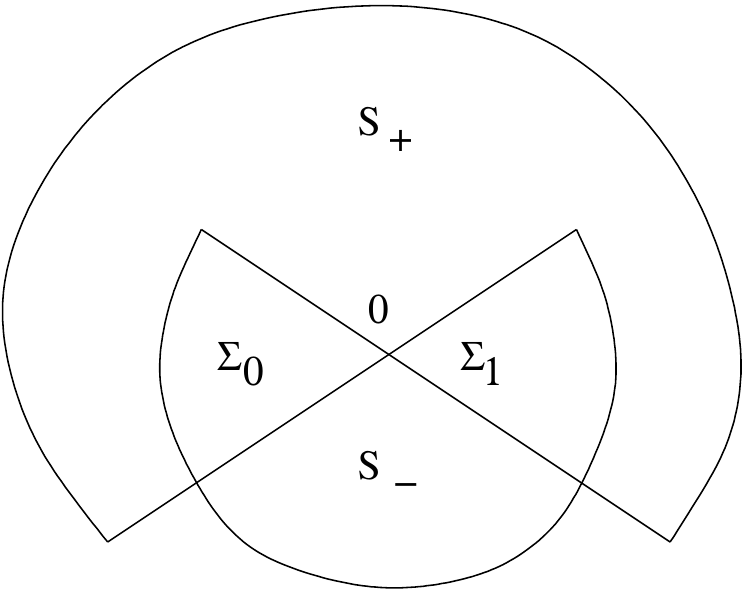}
    \caption{The sectors $S_{\pm}$ defining the Stokes matrices.}
  \end{center}
\end{figure}

The canonical base "at zero" of system (\ref{tty}) in each sector  $S_{\pm}$ is given by the 
fundamental matrix 
\begin{equation}W_{\pm}(z)= H_{\pm}(z)F(z), \ F(z)=
\left(\begin{matrix}  z^{-l}e^{\mu(\frac1z-z)}  & 0\\
 0 & 1 \end{matrix}\right).\label{wpm}\end{equation}
In the definition of the above fundamental matrices we consider that the holomorphic branch of the fundamental 
matrix $F(z)$ of the formal normal form in $S_-$ is obtained from that in $S_+$ by counterclockwise analytic 
extension. We introduce yet another fundamental matrix 
$W_{+,1}(z)=H_+(z)F(z)$ of solutions of 
system (\ref{tty}) on $S_+$ where the holomorphic branch of the matrix $F(z)$ is obtained from that in $S_-$ 
by counterclockwise analytic extension.  
 Let $\Sigma_0$ ($\Sigma_1$) denote the left (respectively, right) component of the intersection 
 $S_{+}\cap S_-$. Over each component $\Sigma_j$ we have two fundamental matrices, 
 $W_{\pm}(z)$ over $\Sigma_0$ and $W_-(z)$, $W_{+,1}(z)$ over $\Sigma_1$ that are  related by a constant 
 transition matrix, the {\it Stokes matrix} $C_j$: 
\begin{equation} W_-(z)=W_+(z)C_0 \text{ on }  \Sigma_0; \ \ W_{+,1}(z)=W_-(z)C_1 \text{ on } \Sigma_1.\label{stokes}\end{equation}

It is well-known that the Stokes matrices are triangular and unipotent:
\begin{equation} C_0=\left(\begin{matrix}  1 & c_0 \\
 0 & 1\end{matrix}\right), \ C_1=\left(\begin{matrix}  1 & 0 \\
 c_1 & 1\end{matrix}\right), \label{stm}\end{equation}
 whenever $\mu>0$, which is our case (if $\mu<0$, then their triangular types are opposite). 
 Their upper (lower) triangular elements  $c_0$, $c_1$ are called the  {\it Stokes multipliers}.
 
Recall that the {\it monodromy operator} 
of a linear differential equation on the Riemann sphere acts on the space of germs of  its solutions at a nonsingular point 
$z_0$. Namely, fix a closed path $\alpha$ starting at $z_0$ in the complement to the singular points of the equation. 
The monodromy operator along the path $\alpha$ sends each germ to the result of its analytic extension along the path $\alpha$. 
It is completely determined by the homotopy class of the path $\alpha$ in the complement to the singular points of the equation. 
Linear system (\ref{tty}) under consideration has  exactly two singular points: zero and 
infinity. By the {\it monodromy operator of system} (\ref{tty}) we mean the monodromy operator along a 
counterclockwise circuit around zero. The monodromy matrix of the formal normal form (\ref{ttyn}) in the canonical 
solution base with a diagonal fundamental matrix is 
\begin{equation}M_N=\diag(e^{-2\pi i l}, 1).\label{mno}\end{equation}
Note that it follows from definition that 
\begin{equation}W_{+,1}(z)=W_+(z)M_N.\label{wnn}\end{equation}
\begin{lemma} \label{lemst} \cite[p. 35]{12}. 
The monodromy matrix of  system (\ref{tty}) in the sectorial canonical solution base $W_+$ equals 
\begin{equation}M=M_NC_1^{-1}C_0^{-1}.\label{monst}\end{equation}
\end{lemma}
 
 \subsection{Reality of the Stokes multipliers on the axes of phase-lock areas}
 \begin{theorem} \label{tre} The Stokes multipliers $c_0$ and $c_1$ of system (\ref{tty}) 
 are real, whenever $l=\frac B{\omega}\in\zz$. 
 \end{theorem}
 \begin{proof} System (\ref{tty}) admits the symmetry 
 \begin{equation}J: (u,v)(z)\mapsto (\hat u,\hat v)(z)= (-\bar u,\bar v)(\bar z),\label{ssymm}\end{equation}
 which induces  $\cc$-antilinear automorphisms of the spaces of germs of its solutions on $\rr_{\pm}$ or equivalently, 
   on $\Sigma_{0,1}$.  Indeed, the right-hand side of system (\ref{tty}) is the vector $(u,v)$ multiplied by 
   $z^{-2}$ times an entire matrix function. The   diagonal terms  of the latter matrix function 
   are real polynomials, while the off-diagonal terms are equal to $\frac z{2i\omega}$. 
   This implies that $J$ is a symmetry of system (\ref{tty}). 
 For every vector function $f=(f_1,f_2)(z)$ by $J_*f$ we will denote its image under the transformation $J$. 
 Note that a vector function $f$ holomorphic on a sector $S_{\pm}$ is transformed to the function $J_*f$ holomorphic 
 on $S_{\mp}$, since the complex conjugation permutes the sectors $S_{\pm}$. 
 Let $f_{j\pm}=(f_{1j,\pm},f_{2j,\pm})$, $j=1,2$, 
 denote the canonical basic sectorial solutions forming the fundamental matrices 
 $W_{\pm}=(f_{ij,\pm}(z))$ in the sectors $S_{\pm}$. 
   \def\det{\operatorname{det}}
\begin{proposition} \label{pmp} One has  $J_*f_{1\pm}=-f_{1\mp}$, whenever $l\in\zz$, and  $J_*f_{2\pm}=f_{2\mp}$ 
in the general case.
\end{proposition}
\begin{proof} Note that $|f_{1\pm}(z)|=o(|f_{2\pm}(z)|)$, as $z\to0$ along  $\rr_-$, by (\ref{wpm}) and since $z^{-l}e^{\frac{\mu}z}\to0$; recall that $\mu>0$, see Convention \ref{conv1}. Therefore, 
$|J_*f_{1+}(z)|=|f_{1+}(\bar z)|=o(|f_{2\pm}(\bar z)|)$, by  (\ref{ssymm}). Thus, both $J_*f_{1+}$ and $f_{1-}$ are solutions 
of system (\ref{tty}) with the least asymptotics, as $z\to0$ along $\rr_-$. 
It is classical that a solution of (\ref{tty}) 
with the least asymptotics  is unique up to constant factor (in particular, $f_{1+}=f_{1-}$ 
on $\rr_-$). In our case this follows by the above asymptotic formula, since the solution 
space is two-dimensional. Therefore, 
$J_*f_{1+}=a_1f_{1-}$, $a_1=const\neq0$. Similarly we get 
that $f_{2+}$ is the solution with the least asymptotics on $\rr_+$, and $J_*f_{2+}=a_2f_{2-}$. 
Let us show that $a_1=-1$ and $a_2=1$. 

The vector function $f_{1\pm}$ is the product of the scalar function $z^{-l}e^{\mu(\frac1z-z)}$ and the first column of the 
matrix function $H_{\pm}(z)$, see (\ref{wpm}). Recall that $H_{\pm}(z)$ is $C^{\infty}$-smooth on $\overline S_{\pm}\setminus\{\infty\}$ and $H_{\pm}(0)=Id$; thus 
the upper element of the latter column tends to 1 and dominates the lower one. Hence, 
$f_{11,\pm}(z)\simeq z^{-l}e^{\mu(\frac1z-z)}$, $f_{21,\pm}(z)=o(f_{11,\pm}(z))$, as $z\to0$ along $\rr_-$. 
If $l\in\zz$, then the first component of the vector 
function $J_*f_{1+}$ equals $-\overline{f_{11,+}(\bar z)}\simeq-z^{-l}e^{\mu(\frac1z-z)}\simeq-f_{11,-}(z)$. Therefore, 
$a_1=-1$. The vector function  
$f_{2\pm}(z)$ equals the second column of the matrix function  $H_{\pm}(z)$,  by (\ref{wpm}). The second 
(lower) component of the latter column tends to 1, as $z\to0$ along $\rr_+$, since $H_{\pm}(0)=Id$. 
The transformation $J$ acts on the second component of a vector function by conjugation of the image and 
the preimage. The three latter statements toghether imply that the second components of the vector functions 
$f_{2-}(z)$ and $J_*f_{2+}(z)$ are asymptotic to each other, as $z\to0$ along $\rr_+$, and hence, $a_2=1$. 
This proves Proposition \ref{pmp}.
\end{proof}

 For every $z\in\rr_-$ one has $f_{1+}(z)=f_{1-}(z)$ (unipotence and upper triangularity of the Stokes matrix $C_0$), hence, 
 $$-\overline{f_{11,\pm}( z)}=-f_{11,\mp}(z)=-f_{11,\pm}(z), \ \overline{f_{21,\pm}(z)}=-f_{21,\mp}(z)=-f_{21,\pm}(z),$$
 by Proposition \ref{pmp}. Therefore,
 $$f_{11,\pm}(z)=f_{11,\mp}(z)\in\rr, \ f_{21,\pm}(z)=f_{21,\mp}(z)\in i\rr  \text{ for }  z\in\rr_-.$$
 Similarly we get that 
\begin{equation} 
f_{12,\pm}(z)=f_{12,\mp}(z)\in i \rr, \ f_{22,\pm}(z)=f_{22,\mp}(z)\in \rr  
\text{ for }  z\in\rr_+,
\label{conj}\end{equation}
\begin{equation} f_{12,-}(z)=-\overline{f_{12,+}(z)}, \ f_{21,+}(z)=-\overline{f_{21,-}(z)} 
\text{ for } z\in\rr,\label{f12-}\end{equation}
by Proposition \ref{pmp}. One has 
$$f_{12,-}(z)=f_{12,+}(z)+c_0f_{11,+}(z)  \text{ for } z\in\rr_-,$$
$$  f_{21,+}(z)=f_{21,-}(z)+c_1f_{22,-}(z)\text{ for } z\in\rr_+,$$
by definition, (\ref{stokes}), (\ref{stm}), (\ref{wnn}) and since in our case, when $l\in\zz$, one has $M_N=Id$ and 
$W_{+,1}=W_+$. Substituting (\ref{f12-}) to the 
latter formulas one gets 
$$c_0f_{11,+}(z)=-2\re f_{12,+}(z)\in\rr \text{ for } z\in\rr_-,$$
$$ c_1f_{22,-}(z)=2\re f_{21,+}(z)\in\rr \text{ for } z\in\rr_+.$$
This together with the reality of the values $f_{11,+}(z)$, $f_{22,-}(z)$ for $z\in\rr_-$ and $z\in\rr_+$ respectively 
and the fact that they do not vanish identically (being the dominant components of non-identically vanishing 
vector functions) implies that $c_0,c_1\in\rr$. Theorem \ref{tre} is proved.
\end{proof}

 \subsection{Symmetries of system (\ref{tty}) and solutions of the special double confluent Heun  equation}
 V.M.Buchstaber and S.I.Tertychnyi have constructed symmetries of double confluent Heun equation (\ref{heun}) \cite{bt1, bt3, tert2}, see also \cite{bt4}. 
The symmetry $\#:E(z)\mapsto 2\omega z^{-l-1}(E'(z^{-1})-\mu E(z^{-1}))$, which is an involution of its solution space, was constructed 
in \cite[equations (32), (34)]{tert2}. It can be obtained from the symmetry $(\phi,t)\mapsto(\pi -\phi,-t)$ of the nonlinear equation (\ref{jos}); the latter symmetry was found in \cite{RK}. The symmetry $\#$ is equivalent to  the symmetry  
 \begin{equation} \bbi: (u,v)(z)\ \mapsto \  (\hat u,\hat v)(z)=-iz^{-l}e^{\mu(\frac 1z- z)}(-v(z^{-1}),u(z^{-1})) 
\label{diez}\end{equation}
of system (\ref{tty}). The transformation 
\begin{equation}\Diamond:E(z)\mapsto e^{\mu(z+z^{-1})}E(-z^{-1})\label{diamond}
\end{equation}
induces an isomorphism of the solution space of the special double confluent Heun equation (\ref{heun})  and that of the "conjugate" double confluent Heun 
 equation (\ref{heun2}). 
Recall that we consider that $l\geq0$. 

As it was shown in \cite{bt1}, equation (\ref{heun}) cannot have polynomial solutions, while (\ref{heun2}) can have them 
only for $l\in\nn$. 
In \cite{bt3} Buchstaber and Tertychnyi  have found new nontrivial symmetries of equation (\ref{heun}) in the case, 
when $l\in\zz_{\geq0}$ and equation (\ref{heun2}) does not have polynomial solutions. 

The following theorem is a direct consequence of results of papers \cite{bt1, 4}. 

\begin{theorem} \label{thec} 
A point $(B,A)\in\rr_{\geq0}\times\rr_+$ is a constriction, if and only if  one of  the following 
equivalent statements holds. 

1) $l=\frac B{\omega}\in\zz$ and system (\ref{tty}) is analytically equivalent at 0 to its formal normal form (\ref{ttyn}). 

2) $l\in\zz$ and $c_0=c_1=0$.

3) System (\ref{tty}) has trivial monodromy. 

4) $l\in\zz$ and $c_0=0$. 

5) $l\in\zz$ and equation (\ref{heun}) has an entire solution. 
\end{theorem}

\begin{proof} Each one of the statements 1), 2),  3) is equivalent to the statement that $(B,A)$ is a constriction, 
see \cite[proposition 3.2, lemma 3.3]{4}. Its equivalence with  
statement 5) was proved in \cite[theorems 3.3, 3.5]{bt1}. It is clear that statement 2) implies 4). 
Let us  prove the converse: statement 4) implies statement 5), which is equivalent to 2). Condition $c_0=0$ (i.e., 
$C_0=Id$) is equivalent to the statement that the second canonical basic solutions $H_{\pm}(z)\left(\begin{matrix} 0\\ 1
\end{matrix}\right)$ of system (\ref{tty}) coming from different sectors paste together and form an entire vector solution 
$(u(z),v(z))$, $v(z)\not\equiv0$, $u(z)=2i\omega zv'(z)$. 
Hence, in this case equation (\ref{heun}) has entire solution $E(z)=e^{\mu z}v(z)$, thus statement 5) holds.
Theorem \ref{thec} is proved.
\end{proof}

\begin{theorem} \label{thpen} \cite[theorem 3.10]{bg}. Let $l\in\nn$, $B=\omega l$, $A>0$. Let equation (\ref{heun2}) 
have a polynomial solution. Then equation (\ref{heun}) has no entire solution; thus $(B,A)$ is not a constriction.
\end{theorem}

\begin{theorem} \label{thpol} Equation (\ref{heun2}) has a nontrivial polynomial solution, if and only if $l\in\nn$, $c_1=0$ 
and $c_0\neq0$. 
\end{theorem}

\begin{proof} Let equation (\ref{heun2}) have a polynomial solution $E(z)$. Then $l\in\nn$, see \cite[section 3]{bt0}, and  
the corresponding equation (\ref{heun}) has 
no entire solution, by Theorem \ref{thpen}. Hence, $c_0\neq0$, by Theorem \ref{thec}. Let us show 
that $c_1=0$. Consider the involution 
$$I:(\phi,t)\mapsto(-\phi,t+\pi),$$
which transforms equation (\ref{jos}) to the same equation with opposite sign at $B$. It induces the "right bemol" 
transformation from  \cite[the formula after (8)]{bt3} (see also \cite[formula (38)]{bt1}) 
sending solutions of equation (\ref{heun}) to solutions of equation (\ref{heun2}).  
It also induces the transformation\footnote{The transformation $-iG_l$ is equivalent to 
the right bemol transformation from \cite[the formula after (8)]{bt3}}
\begin{equation} G_l: (\hat u, \hat v)(z)=z^{l}e^{\mu(\frac1z-z)}(v(-z),u(-z)),\label{bemol}\end{equation}
which sends solutions of system (\ref{tty}) to those of the same system with opposite sign at $l$; the latter system 
(\ref{tty}), where $l$ is taken with the "-" sign, will be referred to, as (\ref{tty})$_-$. And 
vice versa, 
the transformation $G_{-l}$ sends solutions of system (\ref{tty})$_-$ to those 
of (\ref{tty}). It transforms canonical sectorial basic solutions of (\ref{tty})$_-$ in $S_{\pm}$ (appropriately normalized by constant factors) to those of (\ref{tty}) in $S_{\mp}$ and changes the numeration of the canonical basic solutions. This follows from definition and (\ref{wpm}). The polynomial solution $E(z)=e^{\mu z}v(z)$ of equation (\ref{heun2}) is constructed from 
an entire vector solution $h(z)=(u(z), v(z))$ of system (\ref{tty})$_-$, and the latter 
coincides with the second canonical solution of system (\ref{tty})$_-$ in both sectors $S_{\pm}$, as in the proof of Theorem \ref{thec}. Thus, its image $G_{-l}h(z)$, which is holomorphic on $\cc^*$, coincides with
 the first canonical solution of system (\ref{tty}) in both sectors. Hence, the first 
 basic solutions of system (\ref{tty}) coming from different sectors coincide, and thus, 
$C_1=Id$, $c_1=0$, by definition and (\ref{stokes}). 
Vice versa, let $l\in\nn$, $c_1=0$ and $c_0\neq0$. Then the first canonical sectorial solution of system (\ref{tty}) in 
$S_+$ is invariant under the monodromy operator, by (\ref{mno}) and (\ref{monst}), and coincides with 
that coming from the sector $S_-$. Its image under the transformation $G_l$ is the 
second canonical solution of system (\ref{tty})$_-$, the same in both sectors, and hence, it  is holomorphic on $\cc$. 
Let $\hat E(z)=e^{\mu z} v(z)$ be the corresponding solution of equation (\ref{heun2}), which is an entire function. 
Let us show that it is polynomial. Indeed, in the contrary case the operator of taking $l$-th derivative would send it 
to a non-trivial entire solution of equation (\ref{heun}), by \cite[lemma 3]{bt1}, which cannot exist since $c_0\neq0$ and 
by Theorem \ref{thec}. Thus, equation (\ref{heun2}) has a polynomial solution. Theorem \ref{thpol} is proved.
\end{proof}

\subsection{Canonical sectorial solution bases at infinity and the transition matrix "zero-infinity"}

The involution $z\mapsto z^{-1}$ permutes the sectors $S_{\pm}$ and $S_{\mp}$, as does the symmetry with respect to 
the real axis, by assumption, and it fixes each component $\Sigma_{0,1}$ of their intersection. 
The symmetry $\bbi$ of system (\ref{tty}), see (\ref{diez}), 
 induces an isomorphism of the solutions spaces 
of system (\ref{tty}) on the sectors $S_{\pm}$ and $S_{\mp}$ and an automorphism of the two-dimensional space 
of solutions on the component $\Sigma_1$. 

\begin{definition} The {\it canonical bases} of solutions of system (\ref{tty}) in $S_{\pm}$ "at infinity" are those obtained 
from appripriately normalized canonical solution bases at the opposite sectors $S_{\mp}$ by the automorphism $\bbi$; their fundamental matrices are denoted by $\hat W_{\pm}(z)$. In more detail, we set 
\begin{equation} \hat W_-=\bbi (W_+M_N),\label{hatw}\end{equation}
where  the right-hand side is the result of application of the transformation $\bbi$ to the columns of the fundamental matrix under question. 
\end{definition}

\begin{proposition} The canonical bases $W_+M_N$ and $\hat W_-$ in the space of solutions of 
system (\ref{tty}) in $\Sigma_1$ are related by the transition matrix $Q=Q(B,A)$ that is an involution:  
\begin{equation}W_+(z)M_N=\hat W_-(z)Q, \ Q^2=Id.\label{connect}\end{equation}
\end{proposition}
\begin{proof} The automorphism $\bbi$ of the solution space of system (\ref{tty}) is an involution, as is $\#$. It 
 permutes the 
canonical bases with fundamental matrices $W_+(z)M_N$ and $\hat W_-(z)$. This implies that the transition 
matrix $Q$ between them, which is the matrix of the automorphism in the base $\hat W_-(z)$, is also an involution.
\end{proof}

\begin{theorem} \label{tconn} Let $l=\frac B{\omega}\geq0$. 
The matrix $Q=Q(B,A)$ from (\ref{connect}) has the form 
\begin{equation}Q=\left(\begin{matrix} -a & b\\ 
-c & a\end{matrix}\right).\label{qform}\end{equation}
The  coefficients of the matrix $Q$ satisfy the following system of equations: 
\begin{equation} \begin{cases}  a^2=bc+1\\
 a(1-e^{2\pi i l}+c_0c_1)=bc_1e^{-2\pi i l}-cc_0\end{cases},\label{quad}\end{equation}
where $c_0=c_0(B,A)$, $c_1=c_1(B,A)$ are the Stokes multipliers. One has 
\begin{equation} b,c\neq0 \text{ whenever } (B,A) \text{ is a constriction.}\label{inbc}\end{equation}
\end{theorem}
Theorem \ref{tconn} is proved below. Its proof and the proof of the main results of the paper  are based on 
the following key lemma. 

\begin{lemma} \label{lkey} Let $Q=(q_{ij}(B,A))$ be the same, as in (\ref{connect}). If $(B,A)$ is a constriction, then 
$q_{12}(B,A), q_{21}(B,A)\neq0$. 
\end{lemma} 
\begin{proof} Let $f_1=(f_{11},f_{21})(z)$, $f_2=(f_{12},f_{22})(z)$ be the basis of  solutions of system (\ref{tty}) with the fundamental matrix 
$W_+M_N=(f_{ij})$, and let $g_1=(g_{11},g_{21})(z)$, $g_2=(g_{12},g_{22})(z)$ be the basis with the fundamental matrix $\hat W_-=(g_{ij})$.  Let $E_j(z)=e^{\mu z}f_{2j}(z)$, $\hat E_j(z)=e^{\mu z}g_{2j}(z)$ be the solutions 
of equation (\ref{heun}) defined by $f_j$ and $g_j$ respectively. Recall that in our case, when $(B,A)$ is a constriction, one has $l=\frac B{\omega}\in\zz$, by 
Theorem \ref{thec}.  Note that all the functions 
$f_j$, $g_j$, $E_j$, $\hat E_j$ are holomorphic on $\cc^*$, since the monodromy of the system (\ref{tty}) is trivial, by Theorem \ref{thec}. Suppose the contrary: in the corresponding matrix $Q(B,A)$ 
from (\ref{connect}) one has $q_{12}q_{21}=0$. We treate the two following cases separately. 

Case 1): $q_{12}=0$. This means that $f_2=g_2$, $E_2=\hat E_2$ up to constant factor. 
Note that $f_2(z)$ is an entire function, and 
hence, so are $E_2$ and $\hat E_2$ (see the proof of Theorem \ref{thec}). 
On the other hand, 
$$\hat E_2(z)=\# E_2(z)=2\omega z^{-l-1}(E_2'(z^{-1})-\mu E_2(z^{-1})),$$
which follows from the definition of the transformation $\bbi$, see the beginning of 
Subsection 2.3. 
It is clear that the right-hand side of the latter formula tends to zero, as $z\to\infty$, since $E_2$ is holomorphic 
at 0. Finally, $\hat E_2(z)$ is an entire function tending to 0, as $z\to\infty$. Therefore, $\hat E_2=E_2\equiv0$, by Liouville Theorem, 
and the basic function $f_2$ is identically equal to zero. The contradiction thus obtained proves that $q_{12}\neq0$. 

Case 2): $q_{21}=0$. This means that $f_1=c g_1$, $c\equiv const\neq0$, hence $E_1=c\hat E_1$. 
Note that in our case the Stokes matrices are trivial, by Theorem \ref{thec}, and hence, 
system (\ref{tty}) is analytically equivalent to its formal normal form. Therefore, 
the sectorial normalizations $H_{\pm}(z)$ coincide with one and the same matrix function $H(z)$ holomorphic 
at $0$, $H(0)=Id$. The canonical basic solution 
$f_1(z)$  is equal to  $z^{-l}e^{\mu(\frac1z-z)}$ times the first column of the matrix $H(z)$; the upper element 
of the latter column equals 1 at 0 and the lower element vanishes at 0. Hence, 
\begin{equation} E_1(z)=e^{\mu z} f_{21}(z)=z^{1-l}e^{\frac{\mu}z}h(z), 
\text{ as } z\to0,\label{uvf}
\end{equation}
where the function $h(z)=z^{-1}H_{21}(z)$  is holomorphic at $0$. On the other hand,
\begin{equation*}
\hat E_1(z)=(\# E_1)(z)=2\omega z^{-l-1}(E_1'(z^{-1})-\mu E_1(z^{-1}))=O(e^{\mu z}), \text{ as } z\to\infty,\end{equation*}
by (\ref{uvf}). This together with  (\ref{uvf}) and the equality $E_1=c\hat E_1$ 
yields that 
\begin{equation} E_1(z)=O(z^{1-l}e^{\frac{\mu}z}), \text{ as } z\to0; \ E_1(z)=O(e^{\mu z}), \text{ as } z\to\infty.
\label{e1}\end{equation}
 
 The transformation (\ref{diamond}) sends $E_1$ to the solution $\mce$ of the other Heun equation (\ref{heun2}). 
We claim that $\mce$ is a polynomial. Indeed, 
$$\mce(z)=e^{\mu(z+\frac1z)}E_1(-z^{-1}),$$
$$\mce(z)=O(z^{l-1}), \text{ as } z\to\infty, \ \mce(z)=O(1), \text{ as } z\to0,$$
by (\ref{e1}). Thus, $\mce$ is an entire function with at most polynomial growth at infinity, and hence, is a 
polynomial solution of equation (\ref{heun2}). Therefore, equation (\ref{heun}) has no 
entire solution, and  the point $(B,A)$ under consideration is not a constriction, by 
Theorem \ref{thpen}. The contradiction thus obtained proves that $q_{12}, q_{21}\neq0$. Lemma \ref{lkey} is proved.
\end{proof}

\begin{proof} {\bf of Theorem \ref{tconn}.} One has 
$$
Q^2=\left(\begin{matrix} q_{11}^2+q_{12}q_{21} & q_{12}(q_{11}+q_{22})\\
q_{21}(q_{11}+q_{22}) & q_{22}^2+q_{12}q_{21})\end{matrix}\right)=Id,$$
since $Q$ is an involution. Let us prove that $Q$ has the same type, as in (\ref{qform}). 
By the above equation, there are two possible cases (we denote $a=q_{22}$):

a) $a=q_{22}=-q_{11}$, $a^2+q_{12}q_{21}=1$;  

b) $q_{12}=q_{21}=0$, $q_{11}, q_{22}=\pm1$. 

Note that the equations of cases a) and b) are analytic in $(B,A)\in\rr_{\geq0}\times\rr_+$. Each one of the equations $q_{12},q_{21}=0$ 
does not hold whenever 
$(B,A)$ is a constriction (Lemma \ref{lkey}). Hence, these equations are both not satisfied on an open
 subset of points $(B,A)\in\rr_{\geq0}\times\rr_+$. Therefore, equations of case a) hold on the 
latter open subset, and thus, for all $(B,A)\in\rr_{\geq0}\times \rr_+$, by analyticity. This implies that  
the matrix $Q$ has type (\ref{qform}) and satisfies the first relation in (\ref{quad}). 
Statement (\ref{inbc}) of Theorem \ref{tconn} follows from Lemma \ref{lkey}. 
Let us prove the second relation in (\ref{quad}). To this end, we compare the monodromy matrices 
of system (\ref{tty}) in bases $(f_1,f_2)$, $(g_1,g_2)$ represented by the fundamental matrices 
$W_+M_N$ and $\hat W_-$ respectively and write the equation saying that they are conjugated by $Q$. The 
monodromy matrix in the base $W_+M_N$ equals 
\begin{equation}M_+=M_N^{-1}M_NC_1^{-1}C_0^{-1}M_N=C_1^{-1}C_0^{-1}M_N,\label{m+}\end{equation} 
see (\ref{monst}). 
We consider the monodromy operator under question as acting in the space of germs of solutions of system (\ref{tty}) 
at the point $z=1$: a fixed point of the involution $z\mapsto z^{-1}$. 
To calculate the monodromy matrix in the base $(g_1,g_2)$, note that its fundamental matrix $\hat W_-$ is 
obtained from the matrix $W_+M_N$ by the substitution $z\mapsto z^{-1}$, permutation of lines, 
 change of sign of the first line and subsequent multiplication of the whole matrix by  the scalar factor 
 $-iz^{-l}e^{\mu(\frac 1z- z)}$, see (\ref{diez}) and (\ref{hatw}). Permutation of lines and change of sign of one line 
 do not change the monodromy. The substitution $z\mapsto z^{-1}$ changes the monodromy to inverse. 
 As $z$ makes one counterclockwise turn around zero, 
 the latter scalar factor is multiplied by $e^{-2\pi i l}$.  Finally, the monodromy matrix in the base $\hat W_-$ equals 
 \begin{equation}\hat M_-=e^{-2\pi i l}M_+^{-1}=e^{-2\pi i l}M_N^{-1}C_0C_1.\label{hm-}
\end{equation}
On the other hand, $\hat M_-=QM_+Q^{-1}$, by (\ref{connect}). 
Substituting this formula and (\ref{m+}) to (\ref{hm-}) and 
taking into account that $Q^{-1}=Q$ yields
$$M_+Q=Q\hat M_-=C_1^{-1}C_0^{-1}M_NQ=e^{-2\pi i l}QM_N^{-1}C_0C_1.$$
For a matrix $Q$ of type  (\ref{qform}), the latter equation is equivalent to the second relation in (\ref{quad}). 
Theorem \ref{tconn} is proved.
\end{proof} 

\section{Rotation number and trace of monodromy. Proof of  Theorems \ref{locon} and \ref{intcon}}

\subsection{Trace and phase-lock areas}
In what follows we consider the monodromy matrix $M$ of system (\ref{tty}) written in the base $W_+$, see (\ref{monst}). 
Note that $\det M=e^{-2\pi i l}$. We normalize the matrix $M$   by the scalar factor $e^{\pi i l}$ to make it unimodular:
\begin{equation}\wt M=e^{\pi i l}M=e^{\pi il} M_NC_1^{-1}C_0^{-1}=
\left(\begin{matrix} e^{-\pi il} & - e^{-\pi il}c_0\\ -e^{\pi il}c_1 & 
e^{\pi il}(1+c_0c_1)\end{matrix}\right).\label{wtm}\end{equation}
\begin{proposition} The trace of the matrix $\wt M$ is real, that is, $e^{\pi i l}c_0c_1\in\rr$ 
for all the real parameters $(B,A)$ with $A\neq0$. 
\end{proposition}
\begin{proof}
The projectivization of system (\ref{tty}), the Riccati equation (\ref{ric}) is the 
complex extension of a differential equation on the torus: the product of the unit circles in space and time variables. 
This implies that the monodromy of the Riccati equation along the unit circle in the time variable is an automorphism 
of the unit disk, and hence, so is the projectivized monodromy of system (\ref{tty}). Hence, 
\begin{equation}tr\wt M=2\cos(\pi l)+e^{\pi il}c_0c_1\in\rr.\label{trr}\end{equation}
\end{proof}
\begin{proposition} \label{ineql}
A point $(B,A)$ belongs to a phase-lock area, if and only if 
\begin{equation}|tr\wt M|=|2\cos(\pi l)+e^{\pi il}c_0c_1|\geq2;\label{treq}\end{equation}
$(B,A)$ belongs to its interior, if and only if the latter inequality is strict. 
\end{proposition}

\begin{proof} A point $(B,A)$ belongs to a phase-lock area, if and only if the time $2\pi$ flow map of the vector field 
(\ref{josvec}), which acts as a diffeomorphism of the circle 
$S^1=S^1_{\phi}\times 0$, has a fixed point. 
The latter diffeomorphism is the restriction to $S^1=\partial D_1$ of a conformal automorphism of the disk $D_1$: 
the projectivized monodromy of system (\ref{tty}). A conformal automorphism of the disk $D_1$ represented by a 
M\"obius transformation with unimodular matrix has a fixed point in $\partial D_1$, if and only if the trace of the latter 
matrix has module at least two. 
This together with formula (\ref{trr}) proves the first statement of the proposition. The module of the trace equals two, 
if and only if the M\"obius transformation under question is either identity, or parabolic: 
has a unique fixed point in $\partial D_1$. The M\"obius transformation has one of the two 
above types, if and only if the point $(B,A)$ lies in the boundary of a phase-lock area. This follows from strict 
monotonicity of the time $2\pi$ flow map of the vector field (\ref{josvec}) as a function of the parameter $B$. Namely, if the time $2\pi$ flow map $S^1\to S^1$ 
is either identity, or parabolic, 
then slightly deforming $B$ to the right or to the left, one can destroy all its fixed points in 
$S^1=\partial D_1$  and thus, go out of the 
phase-lock area. Conversely, if the point $(B,A)$ lies in the boundary of a phase-lock 
area, then the corresponding above time $2\pi$ flow map is either identity, or parabolic, since 
it is M\"obius and can have obviously 
neither attracting, nor repelling fixed points. Proposition \ref{ineql} is proved.
\end{proof}

\subsection{Behavior of phase-lock areas near constrictions. Proof of Theorems \ref{locon} and \ref{intcon}}
Recall that without loss of generality we consider that $B\geq0$, $A>0$. 

\begin{proof} {\bf of Theorem \ref{locon}.} The monodromy matrix, the matrix $\wt M$, see (\ref{wtm}), and 
\and the Stokes multipliers $c_0$, $c_1$ 
are analytic functions of the parameters $(B,A)$ with $A\neq0$. Fix a constriction $(B_0,A_0)$. Then 
$l=\frac{B_0}{\omega}\in\zz_{\geq0}$, and $c_0(B_0,A_0)=c_1(B_0,A_0)=0$, by Theorem \ref{thec}. Consider  the restrictions of the Stokes multiplier functions to the open ray $B_0\times \rr_+$: 
$$s_j(t):=c_j(B_0,A_0+t),  \ j=0,1, \ t>-A_0; \ s_j(0)=0.$$

{\bf Claim 1.} {\it The difference  $|tr\wt M(B_0,A_0+t)|-2$ is non-zero and has one and the same sign for all $t\neq0$  small enough.}

\begin{proof} The trace under question equals 
$$tr\wt M(B_0,A_0+t)=(-1)^l(2+s_0(t)s_1(t)).$$
The  functions $s_j(t)$ are analytic on the interval $(-A_0,+\infty)$ and do not vanish identically: they may vanish 
only for those $t$, for which $(B_0,A_0+t)$  is a point of intersection of the line $B_0\times\rr$ with 
boundaries of phase-lock areas (Theorems \ref{thec}, \ref{thpol} and \ref{thpol2}). The set of the latter intersection points is discrete, since the boundary of each phase-lock area is the 
graph of a non-constant analytic function $A=g(B)$. Consider the coefficients $a$, $b$, $c$ in the matrix 
$Q=Q(B_0,A_0+t)$, see (\ref{qform}), which are also analytic  functions in $t$.  
In our case, when $l\in\zz$, the second relation in (\ref{quad}) 
with $c_j$ replaced by $s_j(t)$ takes the form 
$$a(t)s_0(t)s_1(t)=b(t)s_1(t)-c(t)s_0(t), \ b(0),c(0)\neq0,$$
by Theorem \ref{tconn}. 
This implies that 
\begin{equation}s_1(t)\simeq\frac {c(0)}{b(0)}s_0(t), \text{ as } t\to0,\label{as01}\end{equation}
$$s_0(t)s_1(t)\simeq\frac{c(0)}{b(0)}s_0^2(t)\simeq qt^{2n}, q\in\rr, \ q\neq0, \ \sign q=\sign\frac{c(0)}{b(0)},$$
\begin{equation}(-1)^ltr\wt M(B_0,A_0+t)=2+qt^{2n}(1+o(1)).\label{astr}\end{equation}
The  latter right-hand side is greater (less) than 2 for all $t\neq0$ small enough, if $q>0$ (respectively, 
$q<0$). This proves the claim.
\end{proof}

There exists a $\var>0$ such that the punctured 
interval $B_0\times([A_0-\var,A_0+\var]\setminus\{ A_0\})$ either lies entirely 
in the interior of a phase-lock area, or lies outside the union of the phase-lock areas. This follows from Proposition \ref{ineql} and 
the above claim. Theorem \ref{locon} is proved.
\end{proof}

Recall that for $r\in\zz$ by  $L_r$ and $\La_r=\{ B=\omega r\}$  we denote respectively the $r$-th phase-lock area and its axis. 
 In the proof of Theorem \ref{intcon} we use the following immediate corollary of the main result of \cite{RK}.

\begin{proposition} \label{33} For every $\omega>0$ and every $r\in\zz_{\geq0}$ there exists an $A\in\rr_+$ that can be chosen 
arbitrarily large such that $(\omega r,A)\in Int(L_r)$. 
\end{proposition}
\begin{proof} The boundary $\partial L_r$ is a union of two 
  graphs of functions $g_{r,\pm}(A)$ 
  with Bessel asymptotics, by the main resulf of \cite{RK}, see formula (\ref{bessas}) at the 
  beginning of the present paper. The functions $h_{\pm}(A)=g_{r,\pm}(A)-r\omega$ tend to zero roughly like sine function multiplied by 
  $A^{-\frac12}$, and $h_{\pm}(A)=-h_{\mp}(A)+o(A^{-\frac12})$, as $A\to\infty$. In particular, there  
  exists a sequence of local maxima $m_k$  of the function $h_{+}$, $m_k\to+\infty$, as $k\to\infty$, such that 
  the values $h_+(m_k)$  are positive  and asymptotically  equivalent to 
  $m_k^{-\frac12}$ times a positive constant  factor. Thus, for 
  every $k$ large enough the values $h_{\pm}(m_k)$ are non-zero and have different 
  signs. This implies that 
   $(\omega r,m_k)\in Int(L_r)$. The proposition is proved.
  \end{proof}
  
\begin{proof} {\bf of Theorem \ref{intcon}.}  Let $r\in\nn$. 
Let $A\in\rr_+$ be such that $P=(r\omega,A)\in  Int(L_r)$, and let $A$ be bigger than the ordinate of every simple intersection in 
$\Lambda_r=\{ B=r\omega\}$ (if any):  
it exists by Proposition \ref{33}. 
Let $I\subset\Lambda_r$ denote the maximal interval containing  $P$ that is contained in 
$L_r$. 

{\bf Claim 2.} {\it The interval $I$ is semi-infinite and bounded from below by a simple 
intersection. Simple intersections do not lie in the $B$-axis.}

\begin{proof}  Suppose the contrary to the first statement of the claim. 
This means that either $I\neq\Lambda_r$ and 
$I$ has a boundary point $E\in\partial L_r$ that is not a simple intersection, or $I=\Lambda_r$. This alternative follows from the fact that 
the interval $I$ cannot be bounded by a simple intersection from above, 
by definition. 
The second case, when $I=\Lambda_r$, is impossible, since  the intersection 
$L_r\cap\{ A=0\}$ 
is one point with abscissa $\sqrt{r^2\omega^2+1}>r\omega$ (see \cite[section 3]{IRF} and  
 \cite[corollary 3]{buch1}),  
which does not lie in $\Lambda_r$.  The latter statement also implies that 
if a simple intersection in $\Lambda_r$ exists, 
it does not lie in the $B$-axis. 
Therefore, the first case takes place. The boundary 
point $E$, which is not a simple intersection, is a constriction, by definition. It is a positive 
constriction, by Theorem \ref{locon} and since $E$ is adjacent to the interval $I\subset L_r$. Therefore, 
a small segment $J$ of the line $\{ B=\omega r\}$ adjacent to $E$ from the other  side is also contained in $L_r$. 
Finally, $I\cup J\subset L_r$, -- a contradiction to the maximality of the interval $I$. The 
claim is proved.
\end{proof}

Claim 2 implies the inclusion $Sr\subset I\subset L_r$ and hence, Theorem \ref{intcon}.
\end{proof}

\section{Open problems and relations between conjectures}

\begin{proposition} Conjecture \ref{coinc} implies Conjectures \ref{cocon} and \ref{right}. Conjecture \ref{right}  implies Conjectures \ref{garl}. 
\end{proposition}

\begin{proof} \
Let us show that Conjecture \ref{right} implies Conjecture \ref{garl}.
Conjecture \ref{right} implies that for every $r\in\nn$ all the constrictions of  the phase-lock 
area $L_r$ have abscissas greater than $(r-1)\omega$. The latter abscissas are equal to $\omega l$, $l\in\zz\cap[0,r]$, 
see \cite{4}. 
Hence, all of them are equal to $\omega r$ and thus,  Conjecture \ref{garl} holds. 

Now let us show that Conjecture \ref{right} follows from Conjecture \ref{coinc}. The equality $L_r^+\cap \La_r=Sr$ 
given by Conjecture \ref{coinc}  implies that 
the set $L_r^+\cap\{ A<A(\mcp_r)\}$ lies on one side from the axis  $\La_r$: either on the right, or on the left. 
It should lie on the right, since the intersection of the phase-lock area $L_r$ with the $B$-axis is just one growth  point 
with the abscissa $B_r=\sqrt{r^2\omega^2+1}>\omega r$ (see  \cite[section 3]{IRF} and \cite[corollary 3]{buch1}), 
hence lying on the right. Finally, the right boundary of the upper phase-lock area $L_r^+$ is a connected curve issued from the point $B_r$ to infinity that lies 
on the right from the axis $\La_r$, meets $\Lambda_r$ at constrictions and 
$\mcp_r$  and separates $\La_r$ from the other upper phase-lock areas $L_k^+$, $k>r$. 
Applying this statement to $r$ replaced by $r-1$ and taking into account symmetry and that $L_0$ contains the $A$-axis 
(by symmetry), we get that $L_r$ lies on the right from the axis $\La_{r-1}$. This proves 
Conjecture \ref{right}. 

Now let us prove that Conjecture \ref{cocon} follows from Conjecture \ref{coinc}. Withour loss of generality 
we treat only the case of 
phase-lock areas $L_r$ with $r\in\nn$ (by symmetry and positivity 
of all the constrictions in $L_0$, which also follows from symmetry). 
 By Conjecture \ref{coinc} and 
the implications proved above, Conjecture \ref{garl} holds: for every $r\in\nn$ all the constrictions 
of the phase-lock area $L_r$ lie in $\La_r$. Hence, those of them with $A>0$ lie in $Sr=L_r^+\cap\La_r$, 
and thus, all of them are 
automatically positive, by connectivity of the ray $Sr$. This together with symmetry implies Conjecture \ref{cocon}. 
 This finishes the proof of the proposition.
\end{proof}

\begin{proposition} \label{reality} For every constriction $(B,A)\in\rr_{\geq0}\times\rr_+$, 
the  ratio  of the triangular 
terms in the corresponding transition matrix $Q$ from (\ref{qform}) is always real. 
\end{proposition}

The proposition follows from formula (\ref{as01}) and reality of the Stokes multipliers of system (\ref{tty}) for $l\in\zz$ 
(Theorem \ref{tre}). 

\begin{conjecture} \label{posit} For every constriction $(B,A)\in\rr_{\geq0}\times\rr_+$  the ratio of the above triangular terms  is negative, that is, $\frac cb>0$. 
\end{conjecture}

\begin{proposition} \label{cb0} Every constriction in $\rr_{\geq0}\times\rr_+$ with $\frac cb>0$ is positive.
\end{proposition}
The proposition follows from formulas (\ref{as01}), (\ref{astr}) and Proposition \ref{ineql}.

\begin{definition} A negative constriction $(B_0,A_0)$ is called {\it queer,} if the germ at $(B_0,A_0)$ of the interior of the corresponding 
phase-lock area lies on one side from the vertical line $\{ B=B_0\}$, and the origin $(0,0)$ lies on the different side.
\end{definition}

\begin{conjecture} \label{queer}  (see a stronger conjecture: \cite[conjecture 5.30]{bg2}). 
There are no queer constrictions. 
\end{conjecture}

\begin{proposition} \label{impl} Conjecture \ref{queer} implies Conjecture \ref{garl}.
\end{proposition}

In the proof of Proposition \ref{impl} we use the following proposition.

\begin{proposition} \label{cominf} 
 Incorrect constrictions 
 cannot come "from infinity". That is, for every $\omega_0\in\rr_+$ and $r\in\nn$
there exists no sequence $\omega_n\to\omega_0$ for which there exists   a sequence of 
constrictions $(B_n,A_n)$ of the phase-lock areas $L_r=L_r(\omega_n)$ with 
 $B_n\neq\omega_nr$ and $A_n\to\infty$.
\end{proposition}

\begin{proof} Without loss of generality we consider that $r\geq3$, since we already know 
that for $r=\pm1,\pm2$ Conjecture \ref{garl}  holds: all the constrictions of the phase-lock area $L_r$ lie in $\Lambda_r$, see Remark \ref{ssylka-4}. 
By symmetry, it suffices to prove the statement of the proposition with $A_n\to+\infty$. 
 For every $r'\in\nn$ the higher simple intersection 
$\mcp_{r'}\in L_{r'}\cap \Lambda_{r'}$ has ordinate $A(\mcp_{r'})$ 
bounded by the maximal root of an explicit family of monic  polynomials depending on the 
parameter $\omega$, 
 see \cite[section 3]{bt0}. Thus, for any given $r'\in\nn$ and  $\omega_0\in\rr_+$ the value $A(\mcp_{r'})$ is 
  locally bounded from above as a function of $\omega$ on a neighborhood of the 
 point $\omega_0$, say $A(\mcp_{r'})<\alpha$ for $r'=r-1$. 
 On the other hand, the ray $R_{\alpha}=\omega r'\times [\alpha,+\infty)\subset Sr'$ lies 
 in $L_{r'}$, by Theorem \ref{intcon}. Therefore, for $r'=r-1$  the  ray $R_{\alpha}$ 
 is  disjoint from the phase-lock area $L_r$ and lies on the left from the intersection $L_r\cap\{ A\geq\alpha\}$ 
   All the possible 
  "incorrect" constrictions of the upper phase-lock area $L_r^+$, that is, the constrictions  
  with abscissas different from $\omega r$ have abscissas no greater than $\omega(r-2)$, 
  see \cite{4} and Remark \ref{ssylka-4}. Hence, they should lie on the 
  left from the line $\Lambda_{r'}=\{ B=\omega r'\}$ containing the ray $R_{\alpha}$. 
  Therefore, their ordinates should be 
  less than $\alpha$. Indeed, if the ordinate of some "incorrect" 
  constriction $X\in L_r^+$ were
   greater or equal to $\alpha$, then $X$ would be separated from the  intersection 
   $L_r\cap\{ A\geq\alpha\}$ by the ray $R_{\alpha}$. Hence, $X$ would not lie in $L_r$, 
   since $L_r\cap\{ A\geq\alpha\}$ is connected (monotonicity of the rotation 
   number function in $B$)
   -- a contradiction. Thus, the  "incorrect" constrictions 
   cannot escape to $+\infty$, as $\omega\to\omega_0$ along some subsequence. 
  This proves Proposition \ref{cominf}.
  \end{proof}

\begin{proof} {\bf of Proposition \ref{impl}.} Here we repeat the arguments from \cite{bg2} preceding conjecture 5.30. Recall that Conjecture \ref{garl} was proved in \cite{4}  for 
$\omega\geq1$: for every $r\in\zz$ 
all the constrictions of the phase-lock area $L_r$ lie in its axis $\La_r$. It was also shown in \cite{4} that 
for arbitrary $\omega>0$ for every $r\in\zz$ the abscissa of each constriction in $L_r$ equals $\omega l$ where $l\in\zz\cap[0,r]$ 
 and $l\equiv r(mod 2)$. Let us now deform $\omega$ from 1 to 0: the phase-lock area portrait will change 
 continuously in $\omega$. Fix an $r\in\nn$. Suppose that there exists  a certain "critical value" $\omega_0\in(0,1)$ 
 such that for every $\omega>\omega_0$ all the constrictions in $L_r$ lie in $\La_r$, and for 
 $\omega=\omega_0$ a new  constriction 
 $(B_0,A_0)$ of the phase-lock area $L_r$ is born and it is "incorrect": it does not lie in $\Lambda_r$. 
 This new constriction can be  born only on an axis $\La_l$ with $0<l<r$, 
 $l\equiv r(mod 2)$.
 Note that this is the only possible scenario of appearence of incorrect constrictions in $L_r$: they cannot 
 come "from infinity", by Proposition \ref{cominf}. Then the boundary of the phase-lock area $L_r$ 
 moves from the right to the left until it touches 
 the axis $\La_l$ for the first time. A point of  intersection $\partial L_r\cap \La_l$ cannot correspond to Heun equation (\ref{heun2}) having a polynomial solution, by Theorem \ref{thpol2} and since $l<r$. 
 Hence, it is a constriction, by the same theorem. 
 Therefore, for $\omega=\omega_0$ the axis $\La_l$ contains a constriction $(B_0,A_0)$ of the phase-lock area $L_r$ 
 and separates $Int(L_r)$ from the origin. See Fig. 9. Thus, the newly born constriction $(B_0,A_0)$ is queer. Therefore, if 
 Conjecture \ref{queer}, which forbids queer constrictions, is true, then we would obtain a contradiction and this would 
 prove Conjecture \ref{garl}. 
 \end{proof}
 
 \begin{figure}[ht]
  \begin{center}
   \epsfig{file=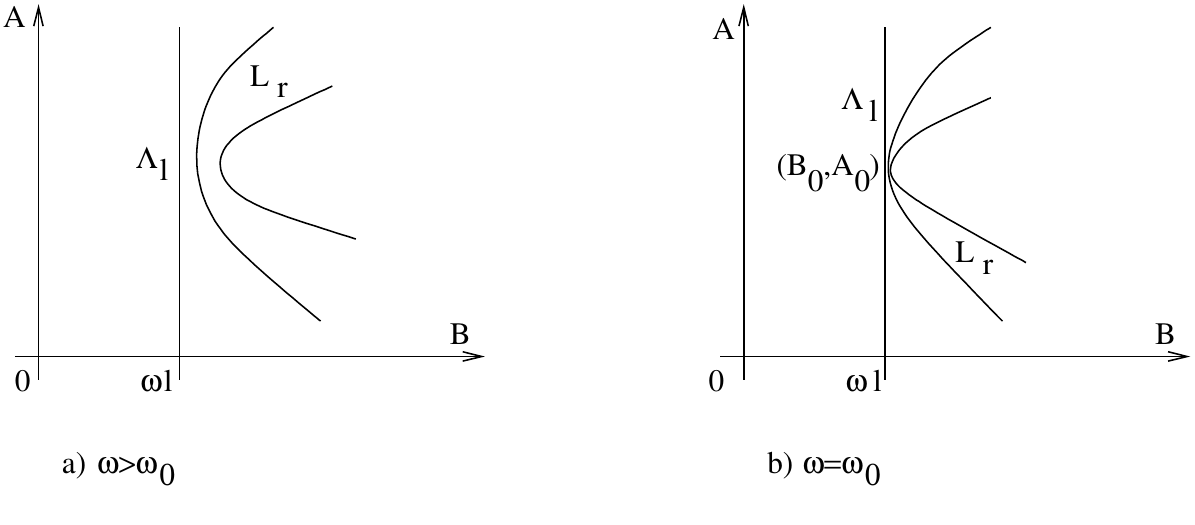}
    \caption{A conjecturally impossible scenario of appearense of an "incorrect" constriction of a phase-lock area $L_r$ in $\La_l$, 
    $0<l<r$, as $\omega$ decreases to the critical value $\omega_0$.}
  \end{center}
\end{figure} 
 
 \begin{proposition} 
Conjecture \ref{posit}  implies Conjecture \ref{cocon}.  Conjecture \ref{cocon} implies Conjecture \ref{queer}. 
Conjecture \ref{queer} implies Conjecture \ref{garl}.
\end{proposition}

\begin{proof} The first implication follows from Proposition \ref{cb0} and the symmetry of the phase-lock area portrait with respect to both coordinate axes. 
Conjecture \ref{cocon} obviously implies Conjecture \ref{queer}, which in its turn implies 
Conjecture \ref{garl}, by Proposition \ref{impl}. 
\end{proof}

\section{Upper triangular element and further properties of the transition matrix $Q$}

\begin{theorem} \label{triq} Let $\omega>0$, $l\geq0$, $\mu>0$, and let 
$$Q=\hat W_-^{-1}(z)W_+(z)M_N=\left(\begin{matrix} -a & b\\ -c & a\end{matrix}\right)$$ 
be the corresponding transition matrix from (\ref{connect}). For every fixed $\omega>0$ and  $l\in\zz$ for every $\mu>0$,  
set $(B,A)=(\omega l, 2\omega\mu)$, one has 
$$b=b(B,A)\in i\rr,$$
$$c=c(B,A)\in i\rr \text{ whenever } (B,A) \text{ is a constriction.}$$  
For every $\omega>0$ and $l\in\nn$ the set of those parameter values $\mu>0$, for which 
$c=c(B,A)\in i\rr$, 
is discrete. 
\end{theorem}

In the proof of Theorem \ref{triq} we use the following formula for the determinant of the fundamental solution matrix $W_+(z)$.

\begin{lemma} \label{detl} For every fundamental matrix $W(z)$ of solutions of system (\ref{tty}) one has 
\begin{equation} \det W(z)=\mathcal S z^{-l}e^{\mu(\frac1z-z)}, \ \mathcal S\equiv const.\label{detw}\end{equation}
The sectorial normalization matrices $H_{\pm}(z)$, see (\ref{wpm}), are unimodular: 
\begin{equation}\det H_{\pm}(z)\equiv1,\label{deth}\end{equation}
\begin{equation} \det W_{\pm}(z)\equiv z^{-l}e^{\mu(\frac1z-z)}.\label{detpm}\end{equation}
Here $W_{\pm}(z)$ are the canonical sectorial fundamental matrices from (\ref{wpm}). 
\end{lemma}
\begin{proof} The proof of formulas (\ref{detw}) and (\ref{deth}) is based on the following Buchstaber--Tertychnyi 
formula for the Wronskian of Heun equation \cite[proof of theorem 4]{bt1}: {\it for every two solutions $E_1$ and $E_2$ of equation (\ref{heun}) 
one has} 
\begin{equation}E'_1(z)E_2(z)-E_2'(z)E_1(z)=\mathcal W z^{-l-1}e^{\mu(z+\frac1z)}, \ \mathcal W\equiv const.
\label{wr}\end{equation}
Let us first prove (\ref{detw}). Fix any pair of vector solutions $f_j(z)=(u_j(z),v_j(z))$ of system (\ref{tty}), $j=1,2$. Let $E_j(z)=e^{\mu z}v_j(z)$ denote 
the corresponding solutions of equation (\ref{heun}). One has $v_j(z)=e^{-\mu z}E_j(z)$, 
$$u_j(z)=2i\omega zv_j'(z)=2i\omega z(e^{-\mu z}E_j(z))'=2i\omega ze^{-\mu z}(E_j'(z)-\mu E_j(z)).
$$
The determinant of the fundamental matrix $W(z)$ of the solutions $f_1(z)$ and $f_2(z)$ of system (\ref{tty}) equals 
$$\det W(z)=u_1(z)v_2(z)-u_2(z)v_1(z)=2i\omega z e^{-2\mu z}(E_1'(z)E_2(z)-E_2'(z)E_1(z))$$
$$=\mathcal S z^{-l}
e^{\mu(\frac1z-z)}, \ \mathcal S=2i\omega\mathcal W,$$
by the  latter formula and (\ref{wr}).  This proves (\ref{detw}). 

Let us now prove (\ref{deth}) and (\ref{detpm}). Let $W_{\pm}(z)=H_{\pm}(z)F(z)$ denote the canonical sectorial fundamental matrix from 
(\ref{wpm}). Its determinant equals $z^{-l}e^{\mu(\frac1z-z)}$ up to multiplicative constant, by (\ref{detw}), 
and the same statement holds for the matrix function $F(z)$:
$$F(z)=\diag(z^{-l}e^{\mu(\frac1z-z)}, 1), \ \det F(z)=z^{-l}e^{\mu(\frac1z-z)}.$$
Therefore, $\det H_{\pm}(z)\equiv const$. The latter constant equals one, since $H_{\pm}(0)=Id$. Formulas 
(\ref{deth}), (\ref{detpm}) and Lemma \ref{detl} are proved.
\end{proof}

\begin{proof} {\bf of Theorem \ref{triq}.} Let $W_+(z)$, $\hat W_-(z)$ be the fundamental matrices of solutions 
of system (\ref{tty}) from (\ref{connect}),
\begin{equation} W_+=\left(\begin{matrix}  f_{11,+} & f_{12,+}\\  f_{21,+} & f_{22,+}\end{matrix}\right),
\label{fwp}\end{equation}
see the notations from the proof of Theorem \ref{tre}, $l\in\zz$. One has $M_N=Id$, since $l\in\zz$, 
hence $W_+=W_+M_N$, 
\begin{equation}\hat W_-(z)=\bbi(W_+)(z)=-iz^{-l}e^{\mu(\frac1z-z)}\left(\begin{matrix}  -f_{21,+} & -f_{22,+} \\  f_{11,+} & f_{12,+}\end{matrix}\right)(z^{-1}),\label{w-1}\end{equation}
by definition, 
\begin{equation} \det W_+(z)=f_{11,+}(z)f_{22,+}(z)-f_{12,+}(z)f_{21,+}(z)=z^{-l}e^{\mu(\frac1z-z)},\label{dwp}
\end{equation}
by (\ref{detpm}). Therefore, 
$$D(z)=\det\left(\begin{matrix}  -f_{21,+} & -f_{22,+} \\  f_{11,+} & f_{12,+}\end{matrix}\right)(z^{-1})$$
$$=(f_{11,+}f_{22,+}-f_{12,+}f_{21,+})(z^{-1})=z^{l}e^{\mu(z-\frac1z)},$$
by (\ref{dwp}). Inverting the matrices in (\ref{w-1}) yields
$$\hat W_-^{-1}(z)=iz^le^{\mu(z-\frac1z)}\left(\begin{matrix} f_{12,+} & f_{22,+}\\ -f_{11,+} & -f_{21,+}\end{matrix}
\right)(z^{-1})D^{-1}(z)$$
\begin{equation}=i\left(\begin{matrix} f_{12,+} & f_{22,+}\\ -f_{11,+} & -f_{21,+}\end{matrix}
\right)(z^{-1}).\label{dw-}\end{equation}
Calculating the product $Q=\hat W_-^{-1}(z)W_+(z)$ for $z=1\in\Sigma_1$ by  formulas (\ref{fwp}) and (\ref{dw-}) 
 yields
 \begin{equation} Q=\left(\begin{matrix} -a & b\\ -c & a\end{matrix}\right), \ b=i(f_{12,+}^2(1)+f_{22,+}^2(1)),\label{qf2}
 \end{equation}
 $$c=i(f_{11,+}^2(1)+f_{21,+}^2(1)), \ a=-i(f_{11,+}(1)f_{12,+}(1)+f_{21,+}(1)f_{22,+}(1)).$$
One has $b\in i\rr$,  by (\ref{qf2}) and since $f_{12,+}(z)\in i\rr$,  $f_{22,+}(z)\in\rr$, whenever $z\in\rr_+$, see 
(\ref{conj}). The first statement of Theorem \ref{triq} is proved. Its second statement, on the lower triangular 
element follows from the first statement and reality of the ratio $\frac bc$ at each constriction, see Proposition 
\ref{reality}. Let us prove the third statement of Theorem \ref{triq}. In our case, when $l\in\nn$, the 
second relation in (\ref{quad})  takes the form 
\begin{equation} bc_1-cc_0=ac_0c_1.\label{qu}\end{equation}
Recall that $c_0$ and $c_1$ are the Stokes multipliers, see (\ref{stm}), and they are real 
for $l\in\zz$, by Theorem \ref{tre}. Writing relation (\ref{qu}) in the new parameters 
$$\beta=-ib, \ \gamma=-ic, \ \sigma=\frac{c_1}{c_0}$$
yields 
$$\beta\sigma-\gamma=-iac_1.$$
Substituting the latter formula to the first equation $a^2=bc+1$ in (\ref{quad}) yields 
\begin{equation}(\beta\sigma-\gamma)^2=-c_1^2a^2=c_1^2(\beta\gamma-1).\label{bga}\end{equation}
This yields the following quadratic equation in $\gamma$:
\begin{equation}\gamma^2-\beta(2\sigma+c_1^2)\gamma+\beta^2\sigma^2+c_1^2=0,\label{qeq}\end{equation}
which has a real solution if and only if its discriminant 
$$\Delta=\beta^2(2\sigma+c_1^2)^2-4(\beta^2\sigma^2+c_1^2)=c_1^2(\beta^2(4\sigma+c_1^2)-4)
$$
is non-negative, or equivalently, 
\begin{equation} \Delta_0=\beta^2(4\sigma+c_1^2)-4\geq0\label{dis0}\end{equation}
(whenever $c_1\neq0$). Note that the coefficients of the transition matrix $Q$ are analytic functions of the parameters on the complement to the hyperplane $\{ A=0\}$. Therefore, 
if, to the contrary, the coefficient $c$ were imaginary (or equivalently, $\gamma\in\rr$) 
on some interval in  $\Lambda_l^+=\{B=l\omega\}\cap\{ A>0\}$, then $\gamma$ 
would be  real  on the whole interval $\Lambda_l^+$, and inequality (\ref{dis0}) 
would hold on $\Lambda_l^+$, since $c_1\not\equiv0$ on $\Lambda_l^+$ 
(see Theorems \ref{thec} and \ref{thpol}). 
But this is impossible: $\Delta_0<0$ on a  neighborhood 
of a simple intersection point $(B,A)\in\Lambda_l^+$, since  $c_1=\sigma=0$, thus, 
$\Delta_0=-4$ at this point. Indeed, a simple intersection corresponds to 
double confluent Heun equation (\ref{heun2}) having a polynomial solution 
(Theorem \ref{thpol2}), and hence, $c_1=0$ there, by 
Theorem \ref{thpol}. The contradiction thus obtained 
implies that $c$ is purely  imaginary on a discrete subset in $\Lambda_l^+$. 
Theorem \ref{triq} is proved.
\end{proof}

Theorem \ref{triq} and the above discussion imply the following corollary. 
\begin{corollary} For every $\omega>0$, $l=\frac B{\omega}\in\nn$ and arbitrary 
$\mu=\frac A{2\omega}>0$ one has 
\begin{equation} \Delta_0=\beta^2(4\sigma+c_1^2)-4\leq0, \text{ where }  
\ \beta=-ib, \ \sigma=\frac{c_1}{c_0}.\label{dis1}\end{equation}
\end{corollary}
\begin{proof} If, to the contrary,  $\Delta_0$ were positive on some interval in $\Lambda_l^+$, then the coefficient $c$ 
would be imaginary there, by the above discussion, -- a contradiction to the last statement of Theorem \ref{triq}. 
\end{proof}

\section{Acknowledgements}

I am grateful to  V.M.Buchstaber for attracting 
my attention to problems on 
model of Josephson effect and helpful discussions. I am grateful to Yu.P.Bibilo for 
helpful discussions.

\end{document}